\documentclass[onefignum,onetabnum]{siamart171218}

\usepackage{amsfonts}
\usepackage{amssymb}
\usepackage{graphicx}
\usepackage{epstopdf}
\usepackage[mathscr]{euscript}
\usepackage{algorithmic}
\ifpdf
  \DeclareGraphicsExtensions{.eps,.pdf,.png,.jpg}
\else
  \DeclareGraphicsExtensions{.eps}
\fi

\usepackage{bm}
\usepackage{tikz-cd}
\usepackage{subfig}
\usepackage{mathtools}
\usepackage{stmaryrd}
\usepackage{accents}
\usepackage{booktabs}

\newsiamremark{remark}{Remark}
\newsiamremark{assumption}{Assumption}

\newsiamremark{hypothesis}{Hypothesis}
\crefname{hypothesis}{Hypothesis}{Hypotheses}
\newsiamthm{claim}{Claim}

\makeatletter
\def\widebreve{\mathpalette\wide@breve}
\def\wide@breve#1#2{\sbox\z@{$#1#2$}%
     \mathop{\vbox{\m@th\ialign{##\crcr
\kern0.08em\brevefill#1{0.8\wd\z@}\crcr\noalign{\nointerlineskip}%
                    $\hss#1#2\hss$\crcr}}}\limits}
\def\brevefill#1#2{$\m@th\sbox\tw@{$#1($}%
  \hss\resizebox{#2}{\wd\tw@}{\rotatebox[origin=c]{90}{\upshape(}}\hss$}
\makeatletter

\newcommand{\triplenorm}[1]{\ensuremath{|\!|\!| #1 |\!|\!|}}

\headers{NZT for Surface Biharmonic}{S. Wu and H. Zhou}

\title{A stabilized nonconforming finite element method for the surface
biharmonic problem
\thanks{
This study is supported in part by the National Natural
Science Foundation of China grant No. 12222101 and the Beijing Natural
Science Foundation No. 1232007.}
}

\author{
Shuonan Wu\thanks{School of Mathematical Sciences, Peking University,
Beijing 100871, China} \and 
Hao Zhou\thanks{School of Mathematical Sciences, Peking University,
Beijing 100871, China.} 
}

\usepackage{amsopn}

\ifpdf
\hypersetup{
  pdftitle={NZT surface},
  pdfauthor={S. Wu and H. Zhou}
}
\fi

\begin{document}

\maketitle

\begin{abstract}
This paper presents a novel stabilized nonconforming finite element method for solving the surface biharmonic problem. The method extends the New-Zienkiewicz-type (NZT) element to polyhedral (approximated) surfaces by employing the Piola transform to establish the connection of vertex gradients across adjacent elements. Key features of the surface NZT finite element space include its $H^1$-relative conformity and weak $\bm{H}({\rm div})$ conformity, allowing for stabilization without the use of artificial parameters. Under the assumption that the exact solution and the dual problem possess only $H^3$ regularity, we establish optimal error estimates in the energy norm and provide, for the first time, a comprehensive analysis yielding optimal second-order convergence in the broken $H^1$ norm. Numerical experiments are provided to support the theoretical results.
\end{abstract}

\begin{keywords}
  surface biharmonic problem; surface nonconforming finite element method; New-Zienkiewicz-type element
\end{keywords}

\begin{AMS}
65N12, 65N15, 65N30
\end{AMS}

\section{Introduction} \label{sec:intro}
Fourth-order partial differential equations (PDEs) on surfaces are widely applied in engineering and physics, including in thin shells \cite{chapelle2011finite}, the surface Cahn-Hilliard equation \cite{elliott2015evolving}, the surface Navier-Stokes equations \cite{reusken2020stream}, and biomembranes \cite{elliott2010modeling,bonito2011dynamics}. 
In this paper, we consider the surface biharmonic problem as follows:
\begin{equation} \label{eq:s-biharmonic}
\Delta_\gamma^2 u = f \quad \text{on }\gamma, \quad \int_\gamma u \mathrm{d}\sigma = 0.
\end{equation}
Here, $\gamma \subset \mathbb{R}^3$ is a compact, closed, and orientable two-dimensional surface without boundary, subject to certain smoothness conditions, and $\Delta_\gamma$ denotes the Laplace-Beltrami operator. The source term $f$ satisfies the compatibility condition $\int_\gamma f \, \mathrm{d}\sigma = 0$. More specific assumptions and notation are given in Section \ref{sec:preliminaries}. 

As a widely utilized discrete method for PDEs on surfaces, surface finite element methods (SFEMs) establish a Galerkin method on the polyhedral (or higher-order) approximated surface $\Gamma_h$ of $\gamma$. Numerous studies have explored the SFEMs for second-order Laplace-Beltrami operator, which is discussed in review articles \cite{dziuk2013finite, bonito2020finite} and their references.

Studies utilizing SFEMs to address fourth-order problems include the use of second-order splitting techniques \cite{elliott2019second, stein2019mixed} and the Hellan-Herrmann-Johnson mixed method, which is specifically designed for solving the surface Kirchhoff plate problem \cite{walker2022kirchhoff}. This latter method can also incorporate Gauss curvature terms to tackle surface biharmonic problem.
Additionally, Larsson and Larson \cite{larsson2017continuous} proposed a scheme that combines continuous piecewise quadratic finite elements with an interior penalty formulation to address jumps in the normal component. In \cite{cai2024continuous}, continuous piecewise linear elements are employed to reconstruct gradients at vertices using weighted average methods, resulting in a continuous piecewise linear reconstructed gradient. In the error estimate of this method, the mesh is assumed to satisfy the $\mathcal{O}(h^2)$-symmetry condition \cite{wei2010superconvergence} to obtain the superconvergence property of the reconstructed gradient.

In the variational formulation of \eqref{eq:s-biharmonic}, we require that the function values and the normal components of its gradient are continuous along any curve, that is, they belong to $H^2(\gamma)$. However, achieving such strong conformity on the discrete surface $\Gamma_h$ is generally not possible. This limitation arises because $\Gamma_h$ is often only Lipschitz continuous, and the varying tangential directions of each discrete element prevent the gradient from maintaining continuity across the edges.

On the other hand, for the planar biharmonic problem, it is common to use the Hessian operator in the variational formulation, which incorporates all second-order derivatives, rather than the Laplace operator. The Hessian operator clearly corresponds to a stronger energy norm, enhancing the stability of the method at both the continuous and discrete levels. However, on general surfaces, using the variational formulations of these two operators introduces additional terms related to Gaussian curvature \cite{reusken2020stream}; see also the discussion in Remark \ref{rk:inapplicable}. As a result, directly substituting the surface Hessian for the Laplace-Beltrami operator is not feasible, creating challenges in designing stable numerical methods.

To address these challenges, in this paper, we extend the New-Zienkiewicz-type (NZT) element to surfaces. The NZT element, introduced by Wang, Shi, and Xu in \cite{wang2007new}, is a class of continuous finite elements with degrees of freedom (DoFs) that include vertex values and vertex gradients. On planar domains, the gradient of the NZT finite element space exhibits weak \(\bm{H}({\rm div})\) properties, ensuring convergence for fourth-order problems. For handling tangential gradient DoFs on the discrete surface \(\Gamma_h\), we adopt a recent approach by Demlow and Neilan that provides a novel perspective on vector-valued nodal DoFs \cite{demlow2024tangential} in solving the surface Stokes problem. Specifically, we modify the NZT finite element space using the inter-element Piola transformation introduced in \cite{demlow2024tangential}, which preserves the weak \(\bm{H}({\rm div}_{\Gamma_h}; \Gamma_h)\) properties of the discrete space. Although adjacent elements on the discrete surface cannot share identical vertex gradient values, it can still be shown that the surface NZT finite element space achieves $H^1$-relative conformity. These properties of the discrete space enable the design of a nonconforming scheme with a parameter-free stabilization.

The proposed nonstandard (nonconforming) finite element space presents several analytical challenges. To establish the approximation properties of the surface NZT finite element space, we carefully analyze the behavior of the tangential gradient and its Piola transformation from $\gamma$ to $\Gamma_h$, identifying the essential property in Lemma \ref{lm:Piola-derivative}, which provides an $\mathcal{O}(h^2)$ approximation of the Piola transformation. For the error estimate, it is sufficient to assume that the solution has $H^3$ regularity, rather than requiring $H^4$ regularity \cite{larsson2017continuous, cai2024continuous}. Furthermore, through detailed analysis, we show that in many error terms, replacing the finite element function with an interpolant from $H^3(\gamma)$ yields a higher-order estimate. Based on these analyses, the dual argument allows us to achieve a new second-order error estimate in the broken $H^1$ norm.


This paper is organized as follows: Section \ref{sec:preliminaries} introduces the preliminaries for the surface operator and its discretization, including the surface Piola transform used for both construction and analysis, as well as the planar NZT element. In Section \ref{sec:FEM}, we present the stabilized nonconforming finite element method based on the surface NZT element for solving \eqref{eq:s-biharmonic}, together with the approximation properties of the finite element space and the stability of the numerical scheme. Section \ref{sec:analysis} demonstrates the optimal error estimates in both the energy norm and the broken $H^1$ norm. Finally, numerical results are provided in Section \ref{sec:numerical}.

We shall use $X \lesssim Y$ (resp. $X \gtrsim Y$) to denote $X \leq CY$ (resp. $X \geq CY$), where $C$ is a constant independent of the mesh size $h$. Additionally, $X \simeq Y$ will signify that both $X \lesssim Y$ and $X \gtrsim Y$ hold.


\section{Notation and preliminaries} \label{sec:preliminaries}

In this paper, we assume that the smoothness of the surface $\gamma \subset \mathbb{R}^3$ is of class $C^4$. This assumption is made solely for analytical convenience, as the proposed algorithm can also be applied to surfaces with lower regularity. Furthermore, we assume that $\gamma$ is compact, closed, and orientable. Under these assumptions, there exists a tubular region $U_\delta = \{ x \in \mathbb{R}^3 : \text{dist}(x, \gamma) < \delta \}$ for sufficiently small $\delta > 0$. Within this region, the signed distance function $d$ is well-defined and of class $C^4$, with $d > 0$ indicating points outside the surface $\gamma$. The unit outward normal vector is defined as $\bm{\nu}(x) := \nabla d(x)$, where $\nabla$ denotes the gradient operator in Euclidean space. Additionally, we denote $\textbf{H}(x) := \nabla^2 d(x)$ as the Weingarten map. For sufficiently small $\delta > 0$, and following the results of Gilbarg and Trudinger \cite[Lemma 14.16]{gilbarg1977elliptic}, the closest point projection $\bm{p}: U_\delta \to \gamma$ is well-defined and given by the formula
\begin{equation} \label{eq:p}
\bm{p}(x) := x - d(x) \bm{\nu}(x).
\end{equation}
The tangential projection operator is defined as $\textbf{P} := \textbf{I} - \bm{\nu} \otimes \bm{\nu}$, where $\textbf{I}$ is the $3 \times 3$ identity matrix and $\otimes$ denotes the outer product of two vectors. It follows that the gradient of the projection can be expressed as
$\nabla \bm{p} = \textbf{P} - d \textbf{H}$.

\subsection{Differential operators and function spaces}
For any scalar function $v$ on $\gamma$, its extension is given as $v^e := v \circ \bm{p}$, which is well-defined on $U_\delta$. Then, the tangential gradient of $v$ on $\gamma$ is given by 
$$
\nabla_{\gamma} v := \textbf{P}\nabla v^e = \nabla v^e - (\nabla v^e \cdot \bm{\nu})\bm{\nu}.
$$ 
For a (column) vector field $\bm{g} = (g_1, g_2, g_3)^T$, we let $\nabla \bm{g}^e = (\nabla g_1^e, \nabla g_2^e, \nabla g_3^e)^T$ denote the Jacobian matrix of $\bm{g}^e$. The surface divergence operator of $\bm{g}$ is defined as
$$
{\rm div}_{\gamma} \bm{g} := \mathrm{tr}(\textbf{P}\nabla \bm{g}^e \textbf{P}) = \nabla \cdot \bm{g}^e - \bm{\nu}^T\nabla \bm{g}^e \bm{\nu}.
$$
The Laplace-Beltrami operator is then defined as $\Delta_\gamma := {\rm div}_\gamma \nabla_\gamma$.

We adopt the standard notation $ W_q^m(\gamma) $ for the Sobolev space of order $ m $ and exponent $ q $ on $ \gamma $, with the corresponding norm given by $ \|\cdot\|_{W^m_q(\gamma)} $. When $ m = 0 $, this space is referred to as $ L^q(\gamma) $. Additionally, we define the Hilbert space $ H^m(\gamma) = W_2^m(\gamma) $ and the $ L^2 $ inner product on $ \gamma $ as $ (\cdot, \cdot)_\gamma $. 
Similar notation applies to any subdomain of $\gamma$.
Furthermore, the subspace of $ W_q^m(\gamma) $ consisting of functions with zero mean is denoted as $ \mathring{W}_q^m(\gamma) $, consistent with similar spaces. We also define the space
$$
\bm{H}({\rm div}_\gamma; \gamma) = \{\bm{g} \in \bm{L}^2(\gamma):~ {\rm div}_\gamma \bm{g} \in L^2(\gamma)\}.
$$

Consider surface $\gamma$ without boundary, Green's formula for tangential differential operators is given by (cf. \cite[Eq. (4.7)]{larsson2017continuous}):
$$
({\rm div}_\gamma \bm{g}, v)_\gamma = -(\bm{g}, \nabla_\gamma v)_\gamma + (\mathrm{tr}(\textbf{H})\bm{\nu} \cdot \bm{g}, v)_\gamma.
$$ 
Since the image of $\nabla_\gamma$ lies in the tangent plane, applying Green's formula twice to the surface biharmonic operator yields
\begin{equation} \label{eq:Green-twice}
(\Delta_\gamma^2 u, v)_\gamma = -(\nabla_\gamma \Delta_\gamma u, \nabla_\gamma v)_\gamma = (\Delta_\gamma u, \Delta_\gamma v)_\gamma. 
\end{equation}

\begin{remark}[inapplicable with $\nabla_\gamma^2$] \label{rk:inapplicable}
When considering the weak form corresponding to \eqref{eq:s-biharmonic}, using $(\nabla_\gamma^2 u, \nabla_\gamma^2 v)_\gamma$ as in the planar case is unjustified. In fact, by applying \cite[Eq. (2.14)]{reusken2020stream}, we obtain
$
(\Delta_\gamma u, \Delta_\gamma v)_\gamma = (\nabla_\gamma^2 u, \nabla_\gamma^2 v)_\gamma - (K\nabla_\gamma u, \nabla_\gamma v)_\gamma,
$
where \( K = K(x) \) is the Gaussian curvature. This clearly shows that the two expressions are not equivalent on a general surface $\gamma$.
\end{remark}

\subsection{Discretization}
Let $\Gamma_h$ be a polyhedral surface approximation of $\gamma$, composed of triangular faces. The surface $\Gamma_h$ provides an $\mathcal{O}(h^2)$ approximation, meaning the distance $d(x)$ between points on $\Gamma_h$ and $\gamma$ satisfies $d(x) = \mathcal{O}(h^2)$. Throughout this paper, we assume that $h$ is {\it sufficiently small} such that $\Gamma_h \subset U_\delta$. This ensures that the closest point projection is well-defined on $\Gamma_h$. Let $\mathcal{T}_h$ be the set of faces of $\Gamma_h$, which is shape-regular and, for simplicity, assumed to be quasi-uniform with $h := \max_{K \in \mathcal{T}_h} \mathrm{diam}(K)$. The images of the mesh elements on the exact surface are given by
$$
K^\gamma := \bm{p}(K) ~~\forall K \in \mathcal{T}_h, \quad \mathcal{T}_h^\gamma := \{\bm{p}(K): K \in \mathcal{T}_h\}.
$$
We denote $\mathcal{E}_h$ the set of edges of $\mathcal{T}_h$, and $\mathcal{V}_h$ be the set of vertices in $\mathcal{T}_h$.
For each $K \in \mathcal{T}_h$, $\mathcal{V}_K$ denote the set of three vertices of $K$. For $e \in \mathcal{E}_h$ with $e = \partial K_1 \cap \partial K_2$, we define the edge patches $\omega_e = K_1\cup K_2$ and $\omega_{e^\gamma} = K_1^\gamma \cup K_2^\gamma$. 

The piecewise constant outward unit normal to $\Gamma_h$ is denoted by $\bm{\nu}_h$. We shall use $\bm{\nu}_K := \bm{\nu}_h|_K$. Under the condition of a sufficiently small mesh size, the estimate $|\bm{\nu} \circ \bm{p} - \bm{\nu}_h| \lesssim h$ holds. For simplicity, when there is no ambiguity, the composition with $\bm{p}$ is sometimes omitted, and we write $|\bm{\nu} - \bm{\nu}_h| \lesssim h$ instead. The tangential projection with respect to $\Gamma_h$ is $\textbf{P}_h := \textbf{I} - \bm{\nu}_h \otimes \bm{\nu}_h$. Let $\mathrm{d}\sigma$ and $\mathrm{d}\sigma_h$ be the surface matures of $\gamma$ and $\Gamma_h$. It holds that $\mathrm{d}\sigma(\bm{p}(x)) = \mu_h(x) \mathrm{d} \sigma_h(x)$, where $\mu_h$ satisfies $|1 - \mu_h| \lesssim h^2$.

\paragraph{Operators and function spaces on $\Gamma_h$} 
Differential operators on $\Gamma_h$, such as $\nabla_{\Gamma_h}$, ${\rm div}_{\Gamma_h}$, and $\Delta_{\Gamma_h}$, can be defined in a similar manner. In the subsequent sections, we assume that these operators are {\it piecewise defined}. Accordingly, we will define piecewise Sobolev spaces:
$$
H_h^m(\Gamma_h) := \{v \in L^2(\Gamma_h): v|_K \in H^m(K),\forall K \in \mathcal{T}_h\}, ~~\|v\|_{H^m_h(\Gamma_h)}^2 := \sum_{K \in \mathcal{T}_h} \|v\|_{H^m(K)}^2.
$$
We denote the $L^2$ inner product on $\Gamma_h$ by $ (\cdot, \cdot)_{\Gamma_h} $. The definitions of the norm and inner product on $K\in \mathcal{T}_h$ and $e\in \mathcal{E}_h$ are similar.

Next, we define some commonly used jump and average operators on $\Gamma_h$.
For $e=\partial K_1^e \cap \partial K_2^e$, let $ \bm{n}_i^e $ (for $ i=1,2 $) denote the in-plane outward unit normal vector with respect to $ \partial K_i^e $ restricted to $e$, and let $ \bm{\tau}^e $ be the unit tangent vector of $ e $. It should be noted that, on the discrete surface $\Gamma_h$, in general, $\bm{n}_1^e \neq - \bm{n}_2^e$. For scalar function $v$ and vector field $\bm{g}$, we define:
\begin{equation} \label{eq:jump-average}
\begin{aligned}
[v] &:=v|_{K_1^e}-v|_{K_2^e},\quad
 \llbracket \bm{g} \rrbracket := \bm{g}|_{K_1^e} - \bm{g}|_{K_2^e},\quad
[\bm{g} \cdot \bm{n} ] := \bm{g} |_{K_1^e}\cdot\bm{n}_1^e + \bm{g} |_{K_2^e}\cdot\bm{n}_2^e,\\
\{v\} &:= \frac{1}{2}(v|_{K^e_1}+v|_{K^e_2}),  \quad 
\{\bm{g} \cdot \bm{n} \} :=\frac{1}{2}(\bm{g} |_{K^e_1}\cdot\bm{n}_1^e-\bm{g} |_{K^e_2}\cdot\bm{n}_2^e).
\end{aligned}
\end{equation}
For a vector field $\bm{g} \in \bm{H}_h^1(\Gamma_h)$, a well-known result states that $\bm{g} \in \bm{H}(\mathrm{div}_{\Gamma_h}; \Gamma_h)$ if and only if $[\bm{g} \cdot \bm{n}]|_e = 0$ for all edges $e$ (cf. \cite{lederer2020divergence}).

\paragraph{Extensions and lifts} For the rest of the paper, we view the closest projection as a mapping from the discrete surface to the the true surface. In this sense, it is a bijection, and its inverse is $\bm{p}^{-1}: \gamma \to \Gamma_h$. If $v$ is defined on $\Gamma_h$, its lift $v^\ell = v \circ \bm{p}^{-1}$. For any $ v \in H_h^m(\gamma) $ (where $ m = 0, 1, 2, 3 $), we have the following norm equivalence:
\begin{equation}  \label{eq:scalar-norm-equiv}
\|v\|_{H^m(K^\gamma)} \simeq \|v^e\|_{H^m(K)} \quad \forall K \in \mathcal{T}_h, \, m = 0, 1, 2, 3.
\end{equation}
This result can be derived using the change of variables and the chain rule (for instance, \cite[pp. 146]{dziuk1988finite} provides a proof for $m = 0,1,2$ on $C^3$ surfaces). It is important to note that on $C^4$ surfaces, the above equivalence holds at most for $m = 3$. For discussions on how to relax the smoothness requirements, we refer to \cite{bonito2020finite}.

Using the transformation relation of tangential gradients on $\gamma$ and $\Gamma_h$, we derive the following integral equality (cf. \cite[Eq. (2.14)]{demlow2009higher}):
\begin{equation}\label{eq:trans-surface}
(\nabla_{\Gamma_h}w^e, \nabla_{\Gamma_h}v^e)_{\Gamma_h} = (\textbf{R}_h\nabla_{\gamma}w, \nabla_{\gamma}v)_\gamma,
\end{equation}
where $\textbf{R}_h : = \mu_h^{-1} \textbf{P}(\textbf{I}-d\textbf{H})\textbf{P}_h(\textbf{I}-d\textbf{H})\textbf{P}$ satisfying $|(\textbf{R}_h - \textbf{I}) \textbf{P}| \lesssim h^2$.

\subsection{Surface Piola transforms}
We note that the derivation of the weak formulation, as shown in \eqref{eq:Green-twice}, requires the tangential derivative $\nabla_\gamma u \in \bm{H}(\mathrm{div}_\gamma; \gamma)$. Based on this consideration, we employ the divergence-conforming Piola transform between surfaces (cf. \cite{cockburn2016hybridizable, bonito2020divergence, lederer2020divergence, demlow2024tangential}).
The general definition is as follows: if $\Phi$ is a diffeomorphism mapping the surface $\mathscr{S}_0$ to the surface $\mathscr{S}_1$, then for a vector field $\bm{g}$ on $\mathscr{S}_0$, the Piola transform $\mathscr{P}_{\Phi}$ maps $\bm{g}$ to $\mathcal{S}_1$, with the expression
$$
(\mathscr{P}_{\Phi}\bm{g}) \circ \Phi := |D\Phi|^{-1} (D\Phi) \bm{g},
$$
where $D\Phi$ is the Jacobian of $\Phi$. If $\mathrm{d}\sigma_i$ represents the surface measure of $\mathscr{S}_i$, then the determinant of $D\Phi$, denoted as $\mu=|D\Phi|$, satisfies $\mu \mathrm{d}\sigma_0 = \mathrm{d}\sigma_1$. Similarly, the Piola transform with respect to $\Phi^{-1}$ can be given as
$$
(\mathscr{P}_{\Phi^{-1}} \bm{g}) \circ \Phi^{-1} := (\mu \circ \Phi^{-1}) (D\Phi^{-1}) \bm{g}, \quad \text{for } \bm{g}: \mathscr{S}_1 \to \mathbb{R}^3.
$$ 
Similar to the Euclidean setting, there holds 
\begin{equation} \label{eq:Piola-div}
{\rm div}_{\mathscr{S}_0} \bm{g} = \mu {\rm div}_{\mathscr{S}_1} \mathscr{P}_{\Phi} \bm{g} \qquad \forall \bm{g} \in \bm{H}({\rm div}_{\mathscr{S}_0}; \mathscr{S}_0).
\end{equation}

In the case of bijection $\bm{p}:\Gamma_h\to \gamma$ (i.e., $\mathscr{S}_0 = \Gamma_h$, $\mathscr{S}_1 = \gamma$ so that $\mu = \mu_h$), the Piola transform of $\bm{g}:\gamma \to \mathbb{R}^3$ with respect to the inverse $\bm{p}^{-1}$ is given by 
\begin{align}\label{def:Piola_invp}
\breve{\bm{g}}:=\mathscr{P}_{\bm{p}^{-1}}\bm{g}=\mu_h\Big[\textbf{I}-\frac{\bm{\nu}\otimes \bm{\nu}_h}{\bm{\nu}\cdot \bm{\nu}_h} \Big][\textbf{I}-d\textbf{H}]^{-1}\bm{g}^e.
\end{align}
Taking $\bm{g} = \nabla_\gamma v$ for some $v \in C^1(\gamma) \cap H_h^2(\gamma)$, then \eqref{eq:Piola-div} implies that 
\begin{equation} \label{eq:Piola-Delta}
\Delta_\gamma v|_{K^\gamma}(\bm{p}(x)) = {\rm div}_\gamma \nabla_\gamma v|_{K^\gamma}(\bm{p}(x)) = \mu_h^{-1}(x) {\rm div}_{\Gamma_h} \widebreve{\nabla_\gamma v}|_K(x) \quad \forall K \in \mathcal{T}_h.
\end{equation}

The following lemma states the equivalence of norms of vector fields and their Piola transforms for $C^4$ surface. 
\begin{lemma}[norm equivalence of the Piola transform]  \label{lm:Piola-equivalence}
For any $K \in \mathcal{T}_h$, let $\bm{g}$ be a vector field on the $C^4$ surface $K^\gamma$, and denote its corresponding Piola transform according to \eqref{def:Piola_invp} by $\breve{\bm g} = \mathscr{P}_{\bm{p}^{-1}} \bm{g}: K \to \mathbb{R}^3$. If $\bm{g} \in \bm{H}^m(K^\gamma)$ for $m=0,1,2$, then $\breve{\bm g} \in \bm{H}^m(K)$. Moreover, we have the following equivalence:
\begin{equation} \label{eq:Piola-equivalence}
\|\bm{g}\|_{H^m(K^\gamma)} \simeq \|\breve{\bm g}\|_{H^m(K)} \quad m=0,1,2.
\end{equation}
\end{lemma}
\begin{proof}
The proof for $m = 0, 1$ is given in \cite[Lemma 4.1]{bonito2020divergence}. The case for $m = 2$ follows a similar argument, so the details will not be elaborated here. We note that since the Piola transform \eqref{def:Piola_invp} involves $\textbf{H}$ with $C^2$ smoothness, the assumption of $C^4$ smoothness is therefore beneficial in the case for $m = 2$.
\end{proof}

We now introduce a key lemma, which establishes a refined approximation of the Piola transform in terms of the tangential derivative. This lemma forms an essential step in constructing our numerical method.
\begin{lemma}[$\mathcal{O}(h^2)$ approximation of Piola transform] \label{lm:Piola-derivative}
For any $v \in C^1(\gamma)$, it holds that
\begin{equation}\label{eq:Piola-derivative}
|\nabla_{\Gamma_h}v^e - \widebreve{\nabla_\gamma v}| \lesssim h^2|(\nabla_\gamma v)^e| \quad \text{on }\Gamma_h. 
\end{equation}
\end{lemma}
\begin{proof}
We have the following relationship
\begin{equation} \label{eq:derivative-ve}
\nabla_{\Gamma_h} v^e 
= \textbf{P}_h \nabla(v \circ \bm{p}) 
= \textbf{P}_h (\textbf{P}-d \textbf{H}) (\nabla_\gamma v)^e.
\end{equation}

Now $|d|+|1-\mu_h|\lesssim h^2$, and $|1-\bm{\nu}_h\cdot \bm{\nu}|=\frac{1}{2}|\bm{\nu}_h-\bm{\nu}|^2\lesssim h^2$. Thus the expression of Piola transformation in \eqref{def:Piola_invp} can be written as 
$$
\begin{aligned}
\widebreve{\nabla_\gamma v} &= \Big( \Big[\textbf{I} - \frac{\bm{\nu}\otimes \bm{\nu}_h}{\bm{\nu}\cdot \bm{\nu}_h}\Big]
+ d \Big[\textbf{I}-\frac{\bm{\nu}\otimes \bm{\nu}_h}{\bm{\nu}\cdot \bm{\nu}_h} \Big]\textbf{H}[\textbf{I}-d\textbf{H}]^{-1}\\
&~~~\qquad\qquad +(\mu_h-1)\Big[\textbf{I}-\frac{\bm{\nu}\otimes \bm{\nu}_h}{\bm{\nu}\cdot \bm{\nu}_h} \Big] [\textbf{I}-d\textbf{H}]^{-1} \Big)(\nabla_\gamma v)^e \\
&= (\textbf{I} - \bm{\nu}\otimes \bm{\nu}_h + \mathcal{O}(h^2)) (\nabla_\gamma v)^e.
\end{aligned}
$$ 
Therefore, we have 
$$
\begin{aligned}
|\nabla_{\Gamma_h} v^e - \widebreve{\nabla_\gamma v}|  
& \lesssim \Big|\Big( [\textbf{I}-\bm{\nu}_h \otimes \bm{\nu}_h][\textbf{I}-\bm{\nu}\otimes \bm{\nu}]-[\textbf{I}- \bm{\nu}\otimes \bm{\nu}_h]+h^2\Big)(\nabla_\gamma v)^e\Big|\\
&=| (-(\bm{\nu}-\bm{\nu}_h)\otimes (\bm{\nu}-\bm{\nu}_h) + h^2 )(\nabla_\gamma v)^e| \lesssim h^2 |(\nabla_\gamma v)^e|,
\end{aligned}
$$
where $|\bm{\nu} - \bm{\nu}_h| \lesssim h$ is applied in the last step. 
\end{proof}

\begin{remark} \label{rm:Piola-vs-extension}
The above lemma indicates that the tangential derivative after applying the Piola transform provides a better approximation to the tangential derivative after extension. This improvement arises intuitively because both derivatives lie in the tangent space of $\Gamma_h$. In fact, if $\nabla_\gamma v$ is directly extended to $\Gamma_h$, from \eqref{eq:derivative-ve} we have only 
\begin{equation} \label{eq:extension-derivative}
|\nabla_{\Gamma_h}v^e - (\nabla_{\gamma} v)^e| \lesssim h |(\nabla_\gamma v)^e|.
\end{equation}
which is less accurate than the Piola transform. Another point is that \eqref{eq:extension-derivative}, combined with the triangle inequality, leads to a pointwise equivalence: $|\nabla_{\Gamma_h}v^e| \simeq |(\nabla_{\gamma} v)^e|$, assuming $h$ is sufficiently small. Note that this assumption, essential to surface FEM, will be taken as a given in further discussions.
\end{remark}

Inspired by \cite{demlow2024tangential}, we apply the (discrete) Piola transforms proposed therein to map between surface triangles. The following definition is given in \cite[Definition 2.3]{demlow2024tangential}: For each vertex $a \in \mathcal{V}_h$, we arbitrary choose a single (fixed) face $K_a \in \mathcal{T}_a$. For $K \in \mathcal{T}_a$, define $\mathscr{M}_a^K: \mathbb{R}^3 \to \mathbb{R}^3$ by 
\begin{equation} \label{def:Mka}
\mathscr{M}_a^K \bm{x} := \Big(\bm{\nu}_{K_a}\cdot \bm{\nu}_K \Big[\bm{I} - \frac{\bm{\nu}_{K_a}\otimes \bm{\nu}_K}{\bm{\nu}_{K_a}\cdot \bm{\nu}_K}\Big] \Big) \bm{x} \quad \forall \bm{x} \in \mathbb{R}^3.
\end{equation}
In particular, $\mathscr{M}_a^K \bm{x}$ is the Piola transform of $\bm{x}$ with respect to the inverse of the closest point projection onto the plane containing $K_a$. 

\begin{lemma}[Lemma 2.5 of \cite{demlow2024tangential}] Fixed $a \in \mathcal{V}_h$, and let $\bm{g}$ lie in the tangent plane of $\gamma$ at $\bm{p}(a)$. For $K \in \mathcal{T}_a$, let $\breve{\bm g}_K := \mathscr{P}_{\bm{p}^{-1}} \bm{g}|_K$ be the Piola transform of $\bm{g}$ to $K$ via the inverse of the closest point projection. Then, 
\begin{equation} \label{eq:Mk}
|\breve{\bm{g}}_K - \mathscr{M}_a^K \breve{\bm{g}}_{K_a}| \lesssim h^2|\breve{\bm g}_{K_a}| \lesssim h^2 |\bm{g}^e|.
\end{equation}
\end{lemma}

By combing \eqref{eq:Piola-derivative} and \eqref{eq:Mk}, with $\bm{g} = \nabla_\gamma v$, we have 
\begin{equation} \label{eq:Piola-Mka}
\Big| (\nabla_{\Gamma_h}v^e)|_K -  \mathscr{M}_a^K \widebreve{\nabla_\gamma v}|_{K_a}\Big| \lesssim h^2 |(\nabla_\gamma v)^e| 
\end{equation}

\subsection{New Zienkiewiz-type element} \label{subsec:NZT}
From the perspective of the space where the variational formulation resides, i.e. \eqref{eq:Green-twice}, the discrete function space requires continuity of the functions and continuity of the tangential gradient in the normal direction. Even in locally flat spaces, these two conditions effectively require the functions to be $C^1$. However, on polyhedral meshes $\Gamma_h$, these continuity requirements conflict at the vertices (unless the local approximated surface is trivial). From this viewpoint, relaxing certain continuity conditions becomes a necessary choice. 

On the other hand, due to the relatively weak nature of the continuous norms in the problem (as discussed in the Remark \ref{rk:inapplicable}), choosing to maintain the continuity of the discrete space (in certain strong sense) while weakly preserving the $\bm{H}({\rm div})$-conformity of the tangential derivatives is a natural option.

Following these principles, a suitable element in the plane is the New Zienkiewicz-type element (NZT element) proposed in \cite{wang2007new}. Although it can be generalized to any dimension, this paper focuses solely on the two-dimensional case ($n=2$). The  shape function space and the definition of degrees of freedom (DoFs) on the triangular element $K$ are given as follows.

\paragraph{Shape function space}
Let $\lambda_i$ \((i=1,2,3)\) be the barycentric coordinates, and define the bubble function $b_K := \lambda_1 \lambda_2 \lambda_3$. For $1\leq i < j \leq 3$, denote
$$
q_{ij} := \lambda_i^2\lambda_j - \lambda_i\lambda_j^2 + \left(2(\lambda_i - \lambda_j) + \frac{3(\nabla\lambda_i-\nabla\lambda_j)\cdot \nabla\lambda_k}{|\nabla \lambda_k|^2}(2\lambda_k -1) \right)b_K \quad k \neq i,j.
$$ 
The space function space is 
\begin{equation} \label{eq:shape-NZT}
V(K) := \mathcal{P}_2(K) + \mathrm{span}\{q_{ij}: 1\leq i < j \leq 3\}.
\end{equation}

\paragraph{Degrees of freedom} The DoFs are defined as the function values and derivative values at the vertices (depicted in Figure \ref{fig:NZT-plane}). For any edge $e$ of $K$, it is noted that the function in $V(K)$ restricted to $e$ is cubic. Therefore, the element (in the planar case) belongs to the class of $C^0$. 

Thanks to the intricately designed shape function space of the NZT element \eqref{eq:shape-NZT}, it satisfies the following key property (see \cite[Equ (7)]{wang2007new}): For any $v \in V(K)$,
\begin{equation}\label{eq:NZT-edge}
  \frac{1}{|e|}  \int_e \frac{\partial v}{\partial \bm{n}^e}\mathrm{d} s = \frac{1}{2} \frac{\partial v}{\partial \bm{n}^e}(a^e_1) + 
  \frac{1}{2} \frac{\partial v}{\partial \bm{n}^e}(a^e_2).
\end{equation}
where $\bm{n}^e$ is the unit out normal of edge $e$ of $K$, $a^e_i~(i=1,2)$ are the vertices of $e$.

\section{Finite element method on surface} \label{sec:FEM}
In this section, we aim to propose an NZT space on a discrete surface and present the corresponding finite element method in conjunction with stabilization techniques.

\subsection{Surface NZT space}
The surface NZT finite element space on the discretized surface $\Gamma_h$ is defined as:
\begin{equation}\label{eq:sNZT-space}
\begin{aligned}
V_h := \big\{v|_K \in V(K):~ v|_K(a) & = v|_{K_a}(a), \\
\nabla_{\Gamma_h}v|_K(a) &= \mathscr{M}_a^{K}(\nabla_{\Gamma_h}v|_{K_a}(a)),~ \forall K\in \mathcal{T}_a,\ \forall a\in \mathcal{V}_h  \big\}.
\end{aligned}
\end{equation}

\begin{figure}[!htbp]
\centering 
\captionsetup{justification=centering}
\subfloat[DoFs of NZT element]{
  \includegraphics[width=0.33\textwidth]{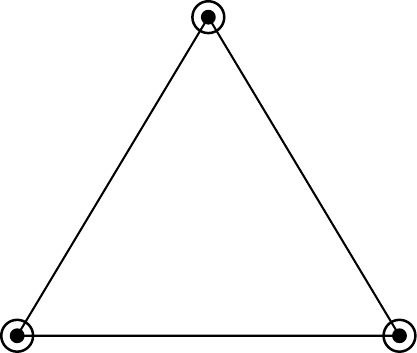}
  \label{fig:NZT-plane}
}\quad %
\subfloat[Connection of vertex gradient DoFs in $V_h$]{
  \includegraphics[width=0.55\textwidth]{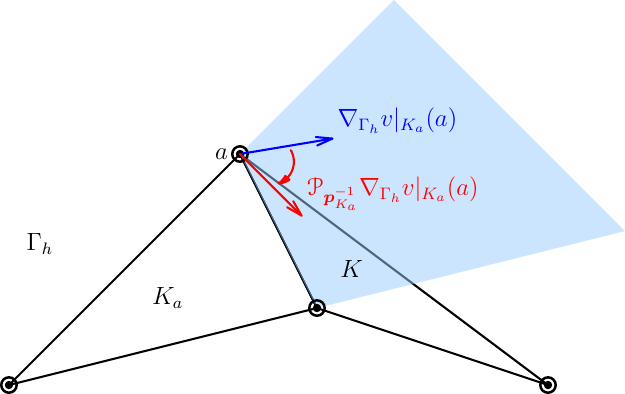}
  \label{fig:div-SAFE1}
}
\caption{Composition of surface NZT space $V_h$ (right) from the NZT element (left), where the DoF of tangential derivative at vertex $a$ in $ K_a $ is mapped to $K$ via the Piola transform with respect to the closest point projection on $ K_a $.}
\end{figure}

Define $\mathring{V}_{h}$ as the integral-free subspace of $V_h$. Some properties of the surface NZT space \eqref{eq:sNZT-space} are considered. First, the discrete tangential gradient is shown to exhibit weak $\bm{H}({\rm div}_{\Gamma_h};\Gamma_h)$ conformity.
\begin{lemma}[weak $\bm{H}({\rm div}_{\Gamma_h}; \Gamma_h)$ conformity of $\nabla_{\Gamma_h}v$]
    For any $v\in V_h$ and $e\in \mathcal{E}_h$, it holds that 
    \begin{align}\label{eq:NZT_D1jumpaverage}
       \frac{1}{|e|} \int_e [\nabla_{\Gamma_h} v \cdot \bm{n}]\mathrm{d} s_h=0.
    \end{align}
\end{lemma}
\begin{proof}
Let $a$ and $b$ be the two vertices of $e$, and let $K_1$ and $K_2$ denote the two elements that share $e$ as a common edge. By the Binet-Cauchy identity, for any $\bm{x} \in \mathbb{R}^3$, 
    $\mathscr{M}_a^{K_j}\bm{x}\cdot \bm{n}_j^e=(\bm{\nu}_{K_a}\times \bm{x})\cdot (\bm{\nu}_{K_j}\times \bm{n}_j^e),$ 
which implies
\begin{align}\label{eq:Mka_div}
\mathscr{M}_a^{K_1}\bm{x}\cdot\bm{n}_1^e+\mathscr{M}_a^{K_2}\bm{x}\cdot\bm{n}_2^e=(\bm{\nu}_{K_a}\times \bm{x})\cdot (\bm{\nu}_{K_1}\times \bm{n}_1^e+\bm{\nu}_{K_2}\times \bm{n}_2^e)=0.
\end{align}
Choosing $\bm{x} = \nabla_{\Gamma_h}v|_{K_a}(a)$ and using the definition of $V_h$ in \eqref{eq:sNZT-space}, it follows that $[\nabla_{\Gamma_h}v \cdot \bm{n}]|_e(a) = 0$. Similarly, $[\nabla_{\Gamma_h}v \cdot \bm{n}]|_e(b) = 0$. Invoking the NZT element property \eqref{eq:NZT-edge}, the desired result is established.
\end{proof}

For the NZT finite element space $V_h$ on discrete surfaces, the vertex derivative values across different elements are determined instead by the discrete Piola transform \eqref{def:Mka}, which results in the loss of $H^1$-conformity. Nevertheless, we will demonstrate that any function in this space is close to an $H^1$-conforming relative. To this end, we first present the following lemma.
\begin{lemma}[local jump estimates] \label{lm:sNZT-jump}
For any $v\in V_h$, it holds that
\begin{subequations} \label{eq:sNZT-jump}
\begin{align}
\big| \llbracket \nabla_{\Gamma_h}v \rrbracket|_e (a) \big|& \lesssim h \big| \nabla_{\Gamma_h} v|_K(a) \big|  \qquad \forall K \in \mathcal{T}_h, e \subset \partial K, a \in \partial e, \label{eq:sNZT-jump1} \\
\| [v] \|_{L^2(e)} &\lesssim h^{\frac{3}{2}} \| \nabla_{\Gamma_h}v \|_{L^2(K)} \quad \forall K \in \mathcal{T}_h, e \subset \partial K. 
\label{eq:sNZT-jump2}
\end{align}
\end{subequations}
\end{lemma}
\begin{proof} 
Applying \eqref{def:Mka}, and noticing that $|\bm{\nu}_K - \bm{\nu}_{K_a}| \lesssim h$, we have
$$
\begin{aligned}
\big|\nabla_{\Gamma_h}v|_K(a)-\nabla_{\Gamma_h}v|_{K_a}(a)\big| &
= \Big|\Big(\bm{\nu}_{K_a}\cdot \bm{\nu}_K \Big[\textbf{I} - \frac{\bm{\nu}_{K_a}\otimes \bm{\nu}_K}{\bm{\nu}_{K_a}\cdot \bm{\nu}_K}\Big]-\textbf{P}_{K_a} \Big) \nabla_{\Gamma_h}v|_{K_a}(a) \Big|\\
&\lesssim h \big|\nabla_{\Gamma_h}v|_{K_a}(a)\big| \qquad \forall a \in \mathcal{V}_h, K \in \mathcal{T}_a.
\end{aligned}
$$
For any $K \in \mathcal{T}_h$, $e \subset \partial K$, $a\in \partial e$, let $e= \partial K \cap \partial K'$. Then, using triangle inequality yields
$$
\begin{aligned}
\big| \llbracket\nabla_{\Gamma_h}v\rrbracket|_e(a) \big| & \leq 
\big| \nabla_{\Gamma_h}v|_K(a) - \nabla_{\Gamma_h}v|_{K_a}(a) \big| +
\big| \nabla_{\Gamma_h}v|_{K'}(a) - \nabla_{\Gamma_h}v|_{K_a}(a) \big| 
 \\
&\lesssim h \big| \nabla_{\Gamma_h}v|_{K_{a}}(a) \big| 
 \lesssim h \big|\nabla_{\Gamma_h}v|_{K}(a)\big|.
\end{aligned}
$$
This proves \eqref{eq:sNZT-jump1}. Let $a_i~(i=1,2)$ be the vertices of $e$.
Note that $[v](a_i) = 0$ and $[v]|_e \in \mathcal{P}_3(e)$, then by applying the standard scaling argument, we obtain \eqref{eq:sNZT-jump2} by 
\begin{align*}
 \| [v] \|_{L^2(e)} &\lesssim h^{\frac{3}{2}} \sum_{i=1,2} \big| \llbracket\nabla_{\Gamma_h}v\rrbracket|_e(a_i) \cdot \bm{\tau}^e \big|
 \leq h^{\frac{3}{2}}\sum_{i=1,2} \big| \llbracket\nabla_{\Gamma_h}v\rrbracket|_e(a_i) \big|\\
 &\lesssim h^{\frac{5}{2}} \| \nabla_{\Gamma_h}v\|_{L^{\infty}(K)}
  \lesssim h^{\frac{3}{2}} \|\nabla_{\Gamma_h}v\|_{L^2(K)}.
\end{align*}
\end{proof}

Below, we define the {\it $H^1$-conforming relative} $\Pi_h^c v$ for any $v \in V_h$. The definition is as follows: for any $K \in \mathcal{T}_h$, $\Pi_h^c v|_K \in V(K)$ and satisfies
\begin{equation}\label{eq:H1-relative}
\Pi_h^c v(a) := v(a), \quad 
(\nabla_{\Gamma_h} \Pi_h^c v \cdot \bm{\tau}^e)(a) := \{\nabla_{\Gamma_h} v \cdot \bm{\tau}^e\}(a) 
\quad \forall a \in \partial e,\ e \subset \partial K.
\end{equation}
Since the NZT element is uniquely determined by the function values and derivative values at the vertices (see subsection \ref{subsec:NZT}), the definition of $\Pi_h^c v$ is therefore well-posed. Moreover, note that $\Pi_h^c v|_e \in \mathcal{P}_3(e)$, so $\Pi_h^c v$ is continuous across any edge $e \in \mathcal{E}_h$, i.e., $\Pi_h^c: V_h \to C^0(\Gamma_h)$.

\begin{lemma}[$H^1$-conforming relative] \label{lm:H1-relative}
 For any $v\in V_h$, the $H^1$-conforming relative \eqref{eq:H1-relative} satisfies
 \begin{equation}\label{eq:sNZT-relative}
     \| v- \Pi_h^c v\|_{L^2(K)} \lesssim h \|v\|_{L^2(K)}, \quad  |v-\Pi_h^c v|_{H^1(K)}\lesssim h|v|_{H^1(K)} \quad \forall K \in \mathcal{T}_h.
 \end{equation}
\end{lemma}
\begin{proof} Denote $v_c := \Pi_h^c v$. For any vertex $a$ of $K$, using the definition of $\Pi_h^c$ in \eqref{eq:H1-relative}, we have 
\begin{align*}
|\nabla_{\Gamma_h}(v-v_c)|_K(a)| 
\lesssim \sum_{e \subset \partial K, \partial e \ni a} \big|\nabla_{\Gamma_h} (v-v_c)
|_K(a) \cdot \bm{\tau}^e\big| 
\leq  \sum_{e \subset \partial K, \partial e \ni a} \big|\llbracket\nabla_{\Gamma_h}v \rrbracket|_e (a)\big|.
\end{align*}
Note that $(v - v_c)|_K \in V(K)$ and $(v - v_c)|_K(a) = 0$ for all $a \in \mathcal{V}_K$. Therefore, we apply \eqref{eq:sNZT-jump1} and the standard scaling arguments to obtain
$$
\begin{aligned}
|v-v_c|_{W^1_{\infty}(K)} &\lesssim \max_{a \in \mathcal{V}_K} \big| \nabla_{\Gamma_h}(v-v_c)|_K(a) \big| 
\lesssim \max_{a \in \mathcal{V}_K} \sum_{e \subset \partial K, \partial e \ni a} \big|\llbracket\nabla_{\Gamma_h}v \rrbracket|_e (a)\big| \\
& \lesssim h \max_{a \in \mathcal{V}_K}  \big|\nabla_{\Gamma_h}v|_K(a) \big| \lesssim h|v|_{W^1_{\infty}(K)}.
\end{aligned}
$$ 
By further using $(v - v_c)|_K(a) = 0$ for $a \in \mathcal{V}_K$ and the inverse inequality, we obtain
\begin{align*}
|v-v_c|_{L^{\infty}(K)}\lesssim h|v-v_c|_{W^1_{\infty}(K)}\lesssim  h^2|v|_{W^1_{\infty}(K)}\lesssim h |v|_{L^{\infty}(K)}.
\end{align*}
By applying the standard scaling argument separately to the two inequalities above, we can derive \eqref{eq:sNZT-relative}.
\end{proof}

\subsection{Approximation properties} Given $w \in C^1(\gamma) \cap H_h^3(\gamma)$, using the definition of $V_h$ in \eqref{eq:sNZT-space} and the surface Piola transform \eqref{def:Piola_invp}, a unique function $\tilde{I}_h w \in V_h$ is defined by
\begin{equation} \label{eq:tildeI}
\tilde{I}_h w(a) = w^e(a), \quad \nabla_{\Gamma_h} \tilde{I}_h w|_{K_a}(a) = \widebreve{\nabla_\gamma w}|_{K_a}(a), \quad \forall a\in \mathcal{V}_h.
\end{equation}
The integral-free subspace projection $I_h w \in \mathring{V}_{h}$ is then defined as
\begin{equation} \label{eq:I}
I_h w := \tilde{I}_h w - \frac{1}{|\Gamma_h|}\int_{\Gamma_h} \tilde{I}_h w \, \mathrm{d}\sigma_h.
\end{equation}

\begin{lemma}[approximation] \label{lm:approximation}
For $w \in C^1(\gamma) \cap H_h^3(\gamma)$, it holds that 
\begin{subequations}
\begin{align} 
\| w^e - \tilde{I}_h w \|_{H^m(K)} &\lesssim h^{3-m} \|w\|_{H^3({K^\gamma})} 
\quad \forall K \in \mathcal{T}_h, ~0\leq m\leq 3, \label{eq:tildeI-L0} \\
\| \widebreve{\nabla_\gamma w} - \nabla_{\Gamma_h} I_h w \|_{H^m(K)} &\lesssim h^{2-m} \| w \|_{H^3({K^\gamma})} 
\quad \forall K \in \mathcal{T}_h, ~0\leq m\leq 2. \label{eq:I-D1}
\end{align}
Moreover, if $w \in C^1(\gamma) \cap H_h^3(\gamma) \cap \mathring{L}^2(\gamma)$, it holds that
\begin{equation} \label{eq:I-L0}
\|w^e - I_h w\|_{H_h^m(\Gamma_h)} \lesssim h^{\min\{3-m,2\}} \|w\|_{H_h^3(\gamma)} \quad 0 \leq m \leq 3.
\end{equation}
\end{subequations}
\end{lemma}
\begin{proof}
Let $z$ be the elementwise interpolant of $w^e$ to $V(K)$, i.e., $z|_K \in V(K)$ and $z|_K(a) = w^e|_K(a)$ and $\nabla_{\Gamma_h}z|_K(a) = \nabla_{\Gamma_h} w^e|_K(a)$ for all $K \in \mathcal{T}_h$ and $a \in \mathcal{V}_K$. By standard approximation theory, this setup gives
\begin{equation} \label{eq:we-z}
\|w^e-z\|_{H^m(K)} \lesssim h^{3-m} | w^e |_{H^3(K)} \lesssim h^{3-m} \|w\|_{H^3(K^\gamma)} \quad \forall K \in \mathcal{T}_h, ~0 \leq m \leq 3.
\end{equation}

Using the property of $\mathscr{M}_a^K$ from \eqref{eq:Piola-Mka} and the definition of $\tilde{I}_h$ in \eqref{eq:tildeI}, we find
$$
\begin{aligned}
|\nabla_{\Gamma_h}(z - \tilde{I}_hw)|_K(a)| &
= \big|\nabla_{\Gamma_h}w^e|_K(a) - \mathscr{M}_a^K \widebreve{\nabla_\gamma w}|_{K_a}(a)\big|\\
&\lesssim h^2 |(\nabla_\gamma w)^e (a)| \quad \forall K \in \mathcal{T}_a.
\end{aligned}
$$ 
Since $(z - \tilde{I}_h w)|_K \in V(K)$ and $(z - \tilde{I}_h w)|_K(a) = 0$ for all $a \in \mathcal{V}_K$, standard scaling arguments the above inequality yield
\begin{equation} \label{eq:z-Ihw}
 \|z - \tilde{I}_hw\|_{H^m(K)} \lesssim h^{2-m} \max_{a \in \mathcal{V}_K}\big| \nabla_{\Gamma_h}(z-\tilde{I}_h w)|_K(a) \big|  \lesssim h^{4-m} \max_{a \in \mathcal{V}_K} |(\nabla_\gamma w)^e (a)|.
\end{equation} 
Using \eqref{eq:extension-derivative}, the inverse estimates and the standard interpolation results yield
\begin{equation} \label{eq:vertex-grad}
\begin{aligned}
\max_{a \in \mathcal{V}_K} |(\nabla_\gamma w)^e (a)| & 
\lesssim \max_{a \in \mathcal{V}_K} \big| \nabla_{\Gamma_h} w^e|_K(a) \big| 
= \max_{a \in \mathcal{V}_K} \big| \nabla_{\Gamma_h} z|_K(a) \big| \\
 & \lesssim \|\nabla_{\Gamma_h} z\|_{L^\infty(K)} \lesssim h^{-1} \|\nabla_{\Gamma_h} z\|_{L^2(K)} \\
 & \leq h^{-1} (\|\nabla_{\Gamma_h} w^e\|_{L^2(K)} + \|\nabla_{\Gamma_h}(w^e - z)\|_{L^2(K)}) \\
 & \lesssim h^{-1} \big(\|w\|_{H^1({K^\gamma})} + h^2 \|w\|_{H^3({K^\gamma})} \big).
 \end{aligned}
\end{equation}
Substituting \eqref{eq:vertex-grad} into \eqref{eq:z-Ihw} and applying \eqref{eq:we-z} gives \eqref{eq:tildeI-L0}.

Let $\bm{I}_1:\bm{C}(K) \to \bm{\mathcal{P}}_1(K)$ denote the standard Lagrange interpolation. Then, for $0 \leq m \leq 2$, by \eqref{eq:Piola-Mka} and the norm equivalence of the Piola transform \eqref{eq:Piola-equivalence}, it holds that
$$
\begin{aligned}
  \|\widebreve{\nabla_\gamma w}-\nabla_{\Gamma_h}w^e\|_{H^m(K)} 
  & \leq  \|\widebreve{\nabla_\gamma w} - \bm{I}_1\widebreve{\nabla_\gamma w}\|_{H^m(K)}
  + \|\bm{I}_1\nabla_{\Gamma_h}w^e-\nabla_{\Gamma_h}w^e\|_{H^m(K)} \\
  &\quad + \|\bm{I}_1\widebreve{\nabla_\gamma w} - \bm{I}_1 \nabla_{\Gamma_h}w^e\|_{H^m(K)}\\
   &\lesssim h^{2-m} \|\widebreve{\nabla_\gamma w}\|_{H^2(K)} 
   + h^{2-m} \|\nabla_{\Gamma_h}w^e\|_{H^2(K)} \\
   &\quad +h^{1-m}\max_{a \in \mathcal{V}_K} \big|\widebreve{\nabla_\gamma w}|_K(a)-\nabla_{\Gamma_h}w^e|_K(a)\big| \\
   &\lesssim h^{2-m}\|w\|_{H^3(K^\gamma)}+h^{3-m}\max_{a \in \mathcal{V}_K} \big| (\nabla_\gamma w)^e|_K(a)\big|.
    \end{aligned}
$$
Combining \eqref{eq:vertex-grad} and \eqref{eq:tildeI-L0} gives \eqref{eq:I-D1}.

If further $w \in \mathring{L}^2(\gamma)$, then by $|1 - \mu_h| \lesssim h^2$, it holds that
$$ 
\begin{aligned}
\|I_h w - \tilde{I}_h w\|_{L^2(\Gamma_h)} &= \frac{1}{|\Gamma_h|^{1/2}} \left|\int_{\Gamma_h} \tilde{I}_h w \mathrm{d} \sigma_h \right| \\
&\lesssim \left|\int_{\Gamma_h} \tilde{I}_h w - w^e \mathrm{d} \sigma_h \right| + \left|\int_{\Gamma_h} w^e \mathrm{d} \sigma_h - \int_\gamma w \mathrm{d}\sigma\right| \\
&\lesssim h^3 \|w\|_{H_h^3(\gamma)} + \int_{\Gamma_h} |w^e(1-\mu_h)| \mathrm{d} \sigma_h \\
&\lesssim h^3 \|w\|_{H_h^3(\gamma)} + h^2 \|w\|_{L^2(\gamma)},
\end{aligned}
$$ 
which gives \eqref{eq:I-L0}.
\end{proof}

Note that $\tilde{I}_h$ and $I_h$ differ by a constant, and the approximation results involving derivatives apply equally to both. The following presents two corollaries of these approximation results.
\begin{corollary}[jump in the normal derivative of interpolant] \label{lm:normal-jump}
For $w \in C^1(\gamma) \cap H_h^3(\gamma)$, it holds that 
\begin{equation} \label{eq:normal-jump}
h_e^{-1/2}\| [\nabla_{\Gamma_h} I_hw \cdot \bm{n}] \|_{L^2(e)} \lesssim h_e \|w\|_{H^3_h(\omega_{e^\gamma})} \quad \forall e \in \mathcal{E}_h.
\end{equation}
\end{corollary}
\begin{proof}
Since $\widebreve{\nabla_\gamma w}\in \bm{H}({\rm div}_{\Gamma_h};\Gamma_h)$, we have $[\widebreve{\nabla_\gamma w} \cdot \bm{n}]|_e = 0$ for all $e \in \mathcal{E}_h$. Therefore, by the standard local trace inequality and \eqref{eq:I-D1}, there holds
$$ 
\begin{aligned}
h_e^{-1/2}&\|[\nabla_{\Gamma_h}I_h w \cdot \bm{n} ]\|_{L^2(e)}
= h_e^{-1/2}\|[(\nabla_{\Gamma_h}I_h w - \widebreve{\nabla_\gamma w}) \cdot \bm{n} ]\|_{L^2(e)} \\
&\lesssim h_e^{-1} \|\nabla_{\Gamma_h}I_h w -\widebreve{\nabla_\gamma w}\|_{L^2(\omega_e)}
+ |\nabla_{\Gamma_h}I_h w - \widebreve{\nabla_\gamma w}|_{H_h^1(\omega_e)} \lesssim h \|w\|_{H^3_h(\omega_{e^\gamma})}.
\end{aligned}
$$  
This proves \eqref{eq:normal-jump}.
\end{proof}

\begin{corollary} \label{co:Delta-interpolation}
For $w \in C^1(\gamma) \cap H_h^3(\gamma)$, it holds that 
\begin{equation} \label{eq:Delta-interpolation}
\|(\Delta_\gamma w)^e - \Delta_{\Gamma_h} I_h w\|_{L^2(K)} \lesssim h \|w\|_{H^3(K)} \quad \forall K \in \mathcal{T}_h.
\end{equation}
\end{corollary}
\begin{proof}
By \eqref{eq:Piola-Delta} and $|1 - \mu_h| \lesssim h^2$, there holds 
$$ 
\|(\Delta_\gamma w)^e - {\rm div}_{\Gamma_h}\widebreve{\nabla_\gamma w}\|_{L^2(K)} \lesssim h^2 \|w\|_{H^2(K^\gamma)}.
$$ 
By further utilizing the approximation result \eqref{eq:I-D1} and triangle inequality, the desired estimate can be obtained.
\end{proof}

\subsection{Stabilized nonconforming FEM}
In this subsection, we introduce the surface finite element method for the biharmonic problem \eqref{eq:s-biharmonic}.  The variational formulation of \eqref{eq:s-biharmonic} seeks $u \in \mathring{H}^2(\gamma)$ such that 
\begin{equation} \label{eq:variational-form}
(\Delta_\gamma u, \Delta_\gamma v)_\gamma = (f, v)_\gamma \quad \forall v \in \mathring{H}^2(\gamma).
\end{equation}
The well-posedness of \eqref{eq:variational-form} can be established using the classical Lax-Milgram Lemma. For $C^4$ surface, employing a partition of unity combined with the interior regularity estimate for elliptic equations \cite[Chapter 6.3.1]{evans2022partial} yields
\begin{equation} \label{eq:regularity}
\|u\|_{H^4(\gamma)} \lesssim \|f\|_{L^2(\gamma)}.
\end{equation}
This regularity result can also be found in \cite{larsson2017continuous, cai2024continuous}, given by \cite[Th. 27]{besse2007einstein}. It is worth noting that, although $H^4$ regularity is achieved on $C^4$ surfaces, only $H^3$ regularity of the solution is utilized in the error estimates due to the constraints imposed by the norm equivalence in \eqref{eq:scalar-norm-equiv}.

We define the bilinear form $a_h: V_h \times V_h \to \mathbb{R}$ as
\begin{equation}\label{eq:ah}
    a_h(w,v):=\sum_{K\in \mathcal{T}_h}\int_K \Delta_{\Gamma_h} w\Delta_{\Gamma_h} v\mathrm{d} \sigma_h +\sum_{e\in \mathcal{E}_h} h_e^{-1}\int_e [\nabla_{\Gamma_h}u \cdot \bm{n}] [\nabla_{\Gamma_h}v \cdot \bm{n}]
    \mathrm{d} s_h.
\end{equation}
It is important to emphasize that no artificial parameter in the stabilization term. 
Such a stabilization is not required for planar problem for the NZT finite element method (or other nonconfomring FEMs) with full second-order derivative.
However, for surface biharmonic problems, since the second-order derivative in the first term of the bilinear form is only $\Delta_{\Gamma_h}$, which is weaker than the full second-order derivative. In this sense, the stabilization in \eqref{eq:ah} becomes necessary.

We define the stabilized nonconforming finite element method as follows: Find $u_h\in \mathring{V}_{h}$ such that
\begin{subequations} 
\begin{equation}\label{eq:FEM}
    a_h(u_h,v_h)=l_h(v_h)\quad \forall v_h\in \mathring{V}_{h},
\end{equation}
where the discrete linear functional is
\begin{equation} \label{eq:source}
    l_h(v_h)=\int_{\Gamma_h} f_h v_h\mathrm{d} \sigma_h \quad \text{with} \quad f_h := f^e-\frac{1}{|\Gamma_h|}\int_{\Gamma_h} f^e \mathrm{d} \sigma_h.
\end{equation}
\end{subequations}

Define the discrete energy semi-norm 
\begin{equation} \label{eq:energy-norm}
\triplenorm{v}_h^2 := \| \Delta_{\Gamma_h} v \|_{L^2(\Gamma_h)}^2 
+ \sum_{e\in \mathcal{E}_h}h_e^{-1} \|[\nabla_{\Gamma_h} v \cdot \bm{n}]\|_{L^2(e)}^2.
\end{equation}
Next, we show that this indeed defines a norm on $\mathring{V}_{h}$. To this end, we first show the discrete Poincar\'e inequality on $\mathring{V}_{h}$.

\begin{lemma}[discrete Poincar\'e inequality] For any $v \in \mathring{V}_{h}$, it holds that 
\begin{equation}\label{eq:poincare}
    \|v\|_{L^2(\Gamma_h)}\lesssim |v|_{H_h^1(\Gamma_h)}.
\end{equation}
\end{lemma}
 \begin{proof}
 Revoking the $H^1$-conforming relative defined in \eqref{eq:H1-relative}, we denote $v_c := \Pi_h^c v \in C^0(\Gamma_h)$, and $\mathring{v}_{c} := v_c - \frac{1}{|\Gamma_h|} \int_{\Gamma_h} v_c \mathrm{d}\sigma_h \in \mathring{H}^1(\Gamma_h)$. Then, the uniform Poincar\'e inequality on $\Gamma_h$ (see \cite[Section 4.2.1]{bonito2020finite}) indicates that 
 $$
 \|\mathring{v}_{c}\|_{L^2(\Gamma_h)} \lesssim |\mathring{v}_{c} |_{H^1(\Gamma_h)} = |v_c|_{H^1(\Gamma_h)},
 $$
 where the hidden constant is independent of $\Gamma_h$.
Using this inequality, combined with the estimates for the $H^1$-conforming relative \eqref{eq:sNZT-relative}, we obtain
\begin{align*}
\|v\|_{L^2(\Gamma_h)} & \leq \|v-v_c\|_{L^2(\Gamma_h)} + \|\mathring{v}_{c}\|_{L^2(\Gamma_h)} 
+ \frac{1}{|\Gamma_h|^{1/2}} \big| \int_{\Gamma_h} v_c \mathrm{d} \sigma_h\big| \\
&\lesssim h \|v\|_{L^2(\Gamma_h)}+ |v_{c}|_{H^1(\Gamma_h)} + \frac{1}{|\Gamma_h|^{1/2}}\big| \int_{\Gamma_h} (v_c - v) \mathrm{d} \sigma_h \big| \\
& \leq h \|v\|_{L^2(\Gamma_h)} + |v|_{H_h^1(\Gamma_h)} + |v - v_c|_{H_h^1(\Gamma_h)} + \|v - v_c\|_{L^2(\Gamma_h)} \\
&\lesssim h \|v\|_{L^2(\Gamma_h)} + (1+h) |v|_{H_h^1(\Gamma_h)},
\end{align*}
where the second inequality uses $\int_{\Gamma_h}v\mathrm{d}\sigma_h=0$. Thus, we have proven \eqref{eq:poincare}, provided that $h$ is appropriately small.
\end{proof}  

\begin{theorem}[energy norm] \label{tm:H1-energy} For any $v \in \mathring{V}_{h}$, it holds that
  \begin{equation}\label{eq:H1-energy}
        \|v\|_{H^1_h(\Gamma_h)} \lesssim \triplenorm{v}_h.
    \end{equation}
\end{theorem}
\begin{proof}Integral by part
$$
\begin{aligned}
|v|^2_{H_h^1(\Gamma_h)}&= \sum_{K \in \mathcal{T}_h} \int_{K} \nabla_{\Gamma_h}v \cdot \nabla_{\Gamma_h}v \mathrm{d}\sigma_h  \\
&= -\sum_{K\in \mathcal{T}_h} \int_K\Delta_{\Gamma_h}v\cdot v\mathrm{d} \sigma_h
+\sum_{e\in \mathcal{E}_h} \int_e [\nabla_{\Gamma_h}v \cdot \bm{n}] \{v\}\mathrm{d}s_h \\
& ~~~+\sum_{e\in \mathcal{E}_h} \int_e \{\nabla_{\Gamma_h}v \cdot \bm{n} \} [v] \mathrm{d} s_h.
\end{aligned}
$$
Next, by the definition of energy (semi-)norm \eqref{eq:energy-norm} and the trace inequality
$$
\begin{aligned}
-\sum_{K\in \mathcal{T}_h} \int_K\Delta_{\Gamma_h}v\cdot v \mathrm{d} \sigma_h & \leq \triplenorm{v}_h \|v\|_{L^2(\Gamma_h)}, \\
\sum_{e\in \mathcal{E}_h} \int_e [\nabla_{\Gamma_h}v \cdot \bm{n}] \{v\}\mathrm{d}s_h & \leq 
\triplenorm{v}_h \left( \sum_{e\in \mathcal{E}_h}h_e \|\{v\}\|_{L^2(e)}^2 \right)^{1/2} \lesssim \triplenorm{v}_h \|v\|_{L^2(\Gamma_h)}.
\end{aligned}
$$
Invoking the jump estimate in \eqref{eq:sNZT-jump2}, there holds
$$
\sum_{e\in \mathcal{E}_h} \int_e \{\nabla_{\Gamma_h} v \cdot \bm{n}\}[v] \mathrm{d} s_h 
\lesssim  \| \nabla_{\Gamma_h} v\|_{L^2(\Gamma_h)}  
\left( \sum_{e \in \mathcal{E}_h} h_e^{-1}\| [v] \|_{L^2(e)}^2 \right)^{1/2} \lesssim  h|v|^2_{H_h^1(\Gamma_h)}.
$$
Combining the above estimates yields
$$
|v|_{H_h^1(\Gamma_h)}^2 \lesssim \triplenorm{v}_h \|v\|_{L^2(\Gamma_h)} + h|v|^2_{H_h^1(\Gamma_h)}.
$$
Applying the discrete Poincar\'e inequality \eqref{eq:poincare} and assuming $h$ is sufficiently small, we then arrive at \eqref{eq:H1-energy}.
\end{proof}

At this point, we have shown that $\triplenorm{\cdot}_h$ is not only a norm on $\mathring{V}_{h}$, but it also controls the broken $H^1$ norm on $\Gamma_h$. By the classical Lax-Milgram Lemma, the nonconforming FEM \eqref{eq:FEM} is well-posed.

\section{Error estimates} \label{sec:analysis}

This section provides error estimates for the energy norm and broken $H^1$ norm of the numerical scheme \eqref{eq:FEM}. The broken $H^1$ norm error estimate based on a duality argument is new, whereas previous work \cite{larsson2017continuous} only considered the $L^2$ error estimate. The estimates in this section primarily follow the classical analysis for nonconforming elements \cite{shi1990error}, while incorporating a refined analysis of the geometric error. It is worth noting that the estimates only require the solution to belong to $H^3(\gamma)$ for both orignial and dual problems.

\subsection{Some auxiliary lemmas} We first present some estimates for the jump and the source terms.
\begin{lemma}[jump estimates] \label{lm:jump-estimate} It holds that 
\begin{subequations} \label{eq:jump-estimate}
\begin{align} 
&\sum_{e\in \mathcal{E}_h} \big|  \langle (\Delta_\gamma w)^e, [\nabla_{\Gamma_h} v_h \cdot \bm{n}]\rangle_e \big| 
\lesssim h \|w\|_{H^3(\gamma)} \triplenorm{v_h}_h  \quad \forall w \in H^3(\gamma), v_h \in V_h, \label{eq:jump-estimate1} \\
& \sum_{e\in \mathcal{E}_h}\big| (\Delta_\gamma w)^e, [\nabla_{\Gamma_h} I_h v \cdot \bm{n}]\rangle_e \big|
\lesssim h^2 \|w\|_{H^3(\gamma)} \|v\|_{H^3(\gamma)} \quad \forall w,v \in H^3(\gamma). \label{eq:jump-estimate2}
\end{align}  
\end{subequations}  
\end{lemma}
\begin{proof}
Let  $P_e^0: L^2(e) \to \mathcal{P}_0(e)$ be the local $L^2$ projection to constant space on edge $e$. Then, by weak $\bm{H}({\rm div}_{\Gamma_h}; \Gamma_h)$ conformity of $V_h$ in \eqref{eq:NZT_D1jumpaverage} and standard approximation theory, we obtain
\begin{align*}
&\sum_{e\in \mathcal{E}_h} \big| \langle (\Delta_\gamma w)^e, [\nabla_{\Gamma_h} v_h \cdot \bm{n}] \rangle_e \big|
= \sum_{e\in \mathcal{E}_h}   \big| \langle (\Delta_\gamma w)^e - P_e^0 (\Delta_\gamma w)^e, [\nabla_{\Gamma_h} v_h \cdot \bm{n}] \rangle_e \big| \\
&\leq  \Big(\sum_{e\in \mathcal{E}_h} h_e \|(\Delta_\gamma u)^e - P_e^0 (\Delta_\gamma u)^e \|_{L^2(e)}^2 \Big)^{1/2} \
          \Big(\sum_{e\in \mathcal{E}_h} h_e^{-1} \|[\nabla_{\Gamma_h} v_h \cdot \bm{n}]\|_{L^2(e)}^2 \Big)^{1/2} \\
&\leq  \Big(\sum_{e\in \mathcal{E}_h} h_e \|(\Delta_\gamma u)^e - P_e^0 (\Delta_\gamma u)^e \|_{L^2(e)}^2 \Big)^{1/2} \triplenorm{v_h}_h \lesssim h \|u\|_{H^3(\gamma)} \triplenorm{v_h}_h.
\end{align*}
Repeat the above estimate by replacing $v_h$ with $I_h v$, and directly applying \eqref{eq:normal-jump}, we then obtain \eqref{eq:jump-estimate2}.
\end{proof}

\begin{lemma}[source estimates] For $f \in \mathring{L}^2(\gamma)$, it holds that 
\begin{subequations} \label{eq:source-estimate}
\begin{align}
|l_h(v_h) - l(\Pi_h^c v_h)^\ell) | & \lesssim h \|f\|_{L^2(\gamma)} \triplenorm{v_h}_h \quad \forall v_h \in \mathring{V}_{h}, \label{eq:source-estimate1} \\
| l_h(I_h v)-l(v)| &\lesssim h^2 \|f\|_{L^2(\gamma)} \|v\|_{H^3(\gamma)} \quad \forall v \in H^3(\gamma). \label{eq:source-estimate2}
\end{align}
\end{subequations}
\end{lemma}

\begin{proof} Revoking the source term in \eqref{eq:source}, we denote $f_h=f^e-\overline{f^e}$ where $\overline{f^e} := \frac{1}{|\Gamma_h|}\int_{\Gamma_h}{f^e}\mathrm{d} \sigma_h$, and $\overline{(\Pi_h^c v_h)^\ell} := \frac{1}{|\gamma|} \int_{\gamma} (\Pi_h^c v_h)^\ell \mathrm{d}\sigma$. Notice that $f \in \mathring{L}^2(\gamma)$, then
$$
\begin{aligned}
l_h(v_h) - l((\Pi_h^c v_h)^\ell) &= (f_h, v_h)_{\Gamma_h} - (f, (\Pi_h^c v_h)^\ell)_\gamma \\
&= (f_h, v_h - \Pi_h^c v_h)_{\Gamma_h} + (f^e - \overline{f^e}, \Pi_h^c v_h)_{\Gamma_h} \\
& ~~~ - (f, (\Pi_h^c v_h)^\ell - \overline{(\Pi_h^c v_h)^\ell})_\gamma \\
&= (f_h, v_h - \Pi_h^c v_h)_{\Gamma_h}  + (f^e - \overline{f^e}, \Pi_h^c v_h - \overline{(\Pi_h^c v_h)^\ell})_{\Gamma_h} \\
& ~~~ - (f - \overline{f^e}, (\Pi_h^c v_h)^\ell - \overline{(\Pi_h^c v_h)^\ell})_\gamma \\ 
\end{aligned}
$$ 
Noticing that the integrand functions in the last two terms happen to be an extension/lift of each other.
Using the property of $H^1$-conforming relative \eqref{eq:sNZT-relative} and $|1 - \mu_h| \lesssim h^2$, we easily deduct \eqref{eq:source-estimate1}.

Next, we denote $\Bar{v}:= \frac{1}{|\gamma|}\int_\gamma v\mathrm{d} \sigma$. Using the similar trick, we have 
$$ 
\begin{aligned}
l_h(I_h v)-l(v) 
=(f_h, I_h v - v^e)_{\Gamma_h} + \big((f- \bar{f^e})^e, (v - \bar{v})^e\big)_{\Gamma_h} - (f- \bar{f^e}, v - \bar{v})_{\gamma}.
\end{aligned}
$$ 
Using the interpolation error estimate \eqref{eq:I-L0} and again the fact that $|1-\mu_h| \lesssim h^2$, 
we easily obtain \eqref{eq:source-estimate2}.
\end{proof}

\subsection{Energy norm error estimates}
\begin{theorem}[energy norm error estimate I] \label{tm:energy-estimate}
Let $u$ be the solution of surface biharmonic problem \eqref{eq:s-biharmonic} on the $C^4$ surface $\gamma$, $u_h \in \mathring{V}_{h}$ be the solution of \eqref{eq:FEM}. Then, 
\begin{equation} \label{eq:energy-estimate}
\triplenorm{I_h u - u_h}_h \lesssim h(\|u\|_{H^3(\gamma)} + \|f\|_{L^2(\gamma)})\lesssim h \| f \|_{L^2(\gamma)}. 
\end{equation}
\end{theorem}
\begin{proof}
Define $u_I = I_h u \in \mathring{V}_h$. Utilizing the coercivity of $a_h(\cdot, \cdot)$ under the norm $\triplenorm{\cdot}_h$, it follows that
$$
\begin{aligned}
\triplenorm{u_I - u_h}_h &\leq \sup_{v_h \in \mathring{V}_h} \frac{a_h(u_I - u_h, v_h)}{\triplenorm{v_h}_h}\\
&\leq \sup_{v_h \in \mathring{V}_h} \frac{|l_h(v_h)-l((\Pi_h^c v_h)^\ell)|}{\triplenorm{v_h}_h}
+  \sup_{v_h \in \mathring{V}_h}\frac{|a_h(u_I, v_h)-(f, (\Pi_h^c v_h)^\ell)_\gamma|}{\triplenorm{v_h}_h} \\
&\lesssim h \|f\|_{L^2(\gamma)} +  \sup_{v_h \in \mathring{V}_h}\frac{|a_h(u_I, v_h)-(f, (\Pi_h^c v_h)^\ell)_\gamma|}{\triplenorm{v_h}_h},
\end{aligned}    
$$
where the source estimate \eqref{eq:source-estimate1} is applied in the last step. With $\Pi_h^c v_h \in H^1(\Gamma_h)$, its lift $(\Pi_h^c v_h)^\ell$ belongs to $H^1(\gamma)$. This allows for decomposing $a_h(u_I, v_h) - (f, (\Pi_h^c v_h)^\ell)_\gamma$ as
$$
\begin{aligned}
&\quad~ a_h(u_I, v_h) - (f, (\Pi_h^c v_h)^\ell)_\gamma \\
&= (\Delta_{\Gamma_h} u_I, \Delta_{\Gamma_h} v_h)_{\Gamma_h} - (\Delta_\gamma^2 u, (\Pi_h^c v_h)^\ell)_\gamma \\
&\quad + \sum_{e \in \mathcal{E}_h} h_e^{-1} \langle [\nabla_{\Gamma_h} u_I \cdot \bm{n} ], [\nabla_{\Gamma_h} v_h \cdot \bm{n} ] \rangle_e \\
&= (\Delta_{\Gamma_h} u_I - (\Delta_\gamma u)^e, \Delta_{\Gamma_h} v_h)_{\Gamma_h} \\
&\quad + ((\Delta_\gamma u)^e, \Delta_{\Gamma_h}v_h)_{\Gamma_h} + (\nabla_{\Gamma_h} (\Delta_\gamma u)^e, \nabla_{\Gamma_h} v_h)_{\Gamma_h} \\
&\quad - (\nabla_{\Gamma_h}(\Delta_\gamma u)^e, \nabla_{\Gamma_h}(v_h - \Pi_h^c v_h))_{\Gamma_h} \\
&\quad - (\nabla_{\Gamma_h}(\Delta_\gamma u)^e, \nabla_{\Gamma_h}\Pi_h^c v_h)_{\Gamma_h} + (\nabla_\gamma (\Delta_\gamma u), \nabla_\gamma (\Pi_h^c v_h)^\ell)_\gamma \\
&\quad + \sum_{e \in \mathcal{E}_h} h_e^{-1} \langle [\nabla_{\Gamma_h} u_I \cdot \bm{n} ], [\nabla_{\Gamma_h} v_h \cdot \bm{n} ] \rangle_e := \sum_{i=1}^5 I_i.
\end{aligned}
$$

Using \eqref{eq:Delta-interpolation}, we have
\begin{equation} \label{eq:energy-I1}
|I_1|  \leq \|\Delta_{\Gamma_h} u_I - (\Delta u)^e\|_{L^2(\Gamma_h)} \|\Delta_{\Gamma_h} v_h \|_{L^2(\Gamma_h)} \\
\lesssim h \|u\|_{H^3(\gamma)}\triplenorm{v_h}_h.
\end{equation}
In light of the jump estimate \eqref{eq:jump-estimate1}, there holds  
\begin{equation} \label{eq:energy-I2}
|I_2| \leq \sum_{e\in \mathcal{E}_h} \big|  \langle (\Delta_\gamma u)^e, [\nabla_{\Gamma_h} v_h \cdot \bm{n}]\rangle_e \big| 
\lesssim h \|u\|_{H^3(\gamma)} \triplenorm{v_h}_h.
\end{equation}
By utilizing the properties of the $H^1$-conforming relation in \eqref{eq:sNZT-relative} and the energy norm to control the $H^1$ norm in \eqref{eq:H1-energy}, it can be deduced that
\begin{equation} \label{eq:energy-I3}
|I_3| \lesssim \|u\|_{H^3(\gamma)} |v_h - \Pi_h^c v_h|_{H_h^1(\Gamma_h)} \lesssim h \|u\|_{H^3(\gamma)}  \triplenorm{v_h}_h.
\end{equation}
From \eqref{eq:sNZT-relative} and \eqref{eq:H1-energy}, we deduce that $\|\Pi_h^c v_h\|_{H^1(\Gamma_h)} \lesssim \triplenorm{v_h}_h$. Further, noting that $\Delta_\gamma u, (\Pi_h^c v_h)^\ell \in H^1(\gamma)$, by \eqref{eq:trans-surface}, we obtain
\begin{equation} \label{eq:energy-I4}
|I_4| = \big|((\textbf{R}_h - \textbf{I})\textbf{P} \nabla_\gamma(\Delta_\gamma u), \nabla_\gamma(\Pi_h^c v_h)^\ell)_\gamma\big| \lesssim h^2 \|u\|_{H^3(\gamma)} \triplenorm{v_h}_h.
\end{equation}
Finally, from the jump of the normal derivative of the interpolant \eqref{eq:normal-jump}, we obtain $|I_5| \lesssim h \|u\|_{H^3(\gamma)} \triplenorm{v_h}_h$. Combining this estimate with \eqref{eq:energy-I1}--\eqref{eq:energy-I4}, we have 
\[
|a_h(u_I, v_h) - (f, (\Pi_h^c v_h)^\ell)_\gamma| \lesssim h \|u\|_{H^3(\gamma)} \triplenorm{v_h}_h.
\]
By substituting the decomposition into the estimate for $\triplenorm{u_I - u_h}_h$ and applying the regularity result \eqref{eq:regularity}, the proof of \eqref{eq:energy-estimate} is completed.
\end{proof}

Using the interpolation estimates \eqref{eq:I-L0} and \eqref{eq:normal-jump}, we ultimately obtain the following theorem.
\begin{theorem}[energy norm error estimate II] \label{tm:energy-estimate2}
Let $u$ be the solution of surface biharmonic problem \eqref{eq:s-biharmonic} on the $C^4$ surface $\gamma$, $u_h \in \mathring{V}_{h}$ be the solution of \eqref{eq:FEM}. Then, 
\begin{subequations}
\begin{align}
    \| (\Delta_\gamma u)^e - \Delta_{\Gamma_h} u_h \|_{L^2(\Gamma_h)} &\lesssim h(\|u\|_{H^3(\gamma)} + \|f\|_{L^2(\gamma)})\lesssim h \| f \|_{L^2(\gamma)}, \\
    \big(\sum_{e\in \mathcal{E}_h}h_e^{-1} \|[\nabla_{\Gamma_h} u_h \cdot \bm{n}]\|_{L^2(e)}^2\big)^{1/2} & \lesssim h(\|u\|_{H^3(\gamma)} + \|f\|_{L^2(\gamma)})\lesssim h \| f \|_{L^2(\gamma)}. \label{eq:energy-estimate-jump}
\end{align}
\end{subequations}
\end{theorem}

\subsection{Broken $H^1$ norm error estimate}
In this section, the duality argument only relies on the regularity result from $\mathring{H}^{-1}(\gamma)$ to $H^3(\gamma)$. For $C^4$ surfaces, this regularity is evident from \eqref{eq:regularity} and standard interpolation theory.
To proceed with the error estimate, we first introduce a lemma that enhances the approximation properties of the interpolated function in the weak form. 

\begin{lemma} \label{lm:bilinear-approximation} For any $w,v\in H^3(\gamma)$, it holds that 
\begin{equation} \label{eq:H1-Delta}
\big| (\Delta_{\Gamma_h} I_h w,\Delta_{\Gamma_h} I_h v)_{\Gamma_h} - (\Delta_\gamma w,\Delta_\gamma v)_\gamma \big|
\lesssim h^2 \|w\|_{H^3(\gamma)} \|v\|_{H^3(\gamma)}. 
\end{equation}
\end{lemma}
\begin{proof} We denote by $w_I = I_h w$, $v_I = I_hv$. Note that $\widebreve{\nabla_\gamma w} \in \bm{H}({\rm div}_{\Gamma_h}; \Gamma_h)$, we write 
$$
\begin{aligned}
&~|(\Delta_{\Gamma_h} w_I - (\Delta_\gamma w)^e ,(\Delta_\gamma v)^e)_{\Gamma_h}|\\
\leq &~\big|(\Delta_{\Gamma_h} w_I -{\rm div}_{\Gamma_h}\widebreve{\nabla_\gamma w},(\Delta_\gamma v)^e)_{\Gamma_h} \big| 
+\big| ({\rm div}_{\Gamma_h}\widebreve{\nabla_\gamma w}-(\Delta_\gamma w)^e,(\Delta_\gamma v)^e)_{\Gamma_h}| \\
\leq &~\big|(\nabla_{\Gamma_h} w_I - \widebreve{\nabla_\gamma w},\nabla_{\Gamma_h}(\Delta_\gamma v)^e)_{\Gamma_h}\big|
+\Big|\sum_{e\in \mathcal{E}_h} \langle (\Delta_\gamma v)^e, [\nabla_{\Gamma_h}w_I \cdot \bm{n}] \rangle_e\Big|\\
& +|({\rm div}_{\Gamma_h}\widebreve{\nabla_\gamma w}-(\Delta_\gamma w)^e,(\Delta_\gamma v)^e)_{\Gamma_h}|.
\end{aligned}
$$
By the approximation result \eqref{eq:I-D1}, jump estimate \eqref{eq:jump-estimate2}, the property of Piola transform \eqref{eq:Piola-Delta} and $|1 - \mu_h| \lesssim h^2$, we deduce that 
\begin{equation} \label{eq:H1-Delta-err1}
|(\Delta_{\Gamma_h} w_I - (\Delta_\gamma w)^e ,(\Delta_\gamma v)^e)_{\Gamma_h}| \lesssim h^2 \|w\|_{H^3(\gamma)} \|v\|_{H^3(\gamma)}.
\end{equation}

In a similar way, we have $|((\Delta_\gamma w)^e ,(\Delta_\gamma v)^e - \Delta_{\Gamma_h} v_I)_{\Gamma_h}| \lesssim h^2 \|w\|_{H^3(\gamma)} \|v\|_{H^3(\gamma)}$. Combine the above estimates with \eqref{eq:Delta-interpolation} and $|1-\mu_h| \lesssim h^2$, we have 
$$
\begin{aligned}
&~|(\Delta_{\Gamma_h} w_I, \Delta_{\Gamma_h}v_I)_{\Gamma_h} - (\Delta_\gamma w,\Delta_\gamma v)_\gamma| \\
\leq &~ |  (\Delta_{\Gamma_h} w_I - (\Delta_\gamma w)^e, \Delta_{\Gamma_h}v_I - (\Delta_\gamma w)^e)_{\Gamma_h}| \\
& + |(\Delta_{\Gamma_h} w_I - (\Delta_\gamma w)^e ,(\Delta_\gamma v)^e)_{\Gamma_h}| 
   + |((\Delta_\gamma w)^e ,(\Delta_\gamma v)^e - \Delta_{\Gamma_h} v_I)_{\Gamma_h}| \\
& + |((\Delta_\gamma w)^e,(\Delta_\gamma v)^e)_{\Gamma_h}- (\Delta_\gamma w,\Delta_\gamma v)_\gamma| \\
\lesssim &~  h^2 \|w\|_{H^3(\gamma)} \|v\|_{H^3(\gamma)}.
\end{aligned}
$$
This completes the proof. 
\end{proof}

\begin{theorem}[broken $H^1$ norm error estimate]\label{tm:H1-estimate} Let $u$ be the solution of surface biharmonic problem \eqref{eq:s-biharmonic} on the $C^4$ surface $\gamma$, $u_h \in \mathring{V}_{h}$ be the solution of \eqref{eq:FEM}. It holds that 
\begin{equation}\label{eq:H1-estimate}
    \| u^e - u_h\|_{H_h^1(\Gamma_h)} \lesssim h^2 ( \|u\|_{H^3(\gamma)} + 
    \| f \|_{L^2(\gamma)}) \lesssim h^2 \| f \|_{L^2(\gamma)}.
\end{equation}
\end{theorem}
\begin{proof}
We denote $u_I =I_h u$ and $\epsilon_h :=u_I-u_h\in V_h$. Note that the $H^1$-conforming relative, $(\Pi_h^c \epsilon_h)^\ell \in H^1(\gamma)$, then $g:=-\Delta_\gamma(\Pi_h^c \epsilon_h)^\ell \in H^{-1}(\gamma)$ and $\int_\gamma g \mathrm{d}\sigma = 0$. Consider the following auxiliary problem: 
$$
\Delta_\gamma^2 \phi = g \quad \text{on } \gamma, \quad \int_{\gamma} \phi \mathrm{d}\sigma = 0,
$$
which satisfies $\|\phi\|_{H^3(\gamma)} \lesssim \|g\|_{\mathring{H}^{-1}(\gamma)}$ due to the regularity result. 
For any $v \in \mathring{H}^1(\gamma)$, by Green's formula
$$ 
(g,v)_\gamma = -(\Delta_\gamma(\Pi_h^c \epsilon_h)^\ell, v)_\gamma
 = (\nabla_\gamma(\Pi_h^c \epsilon_h)^\ell,\nabla_\gamma v)_\gamma 
\leq |(\Pi_h^c \epsilon_h)^\ell|_{H^1(\gamma)} \|v\|_{H^1(\gamma)},
$$
which gives 
\begin{equation} \label{eq:regularity-bound}
\|\phi\|_{H^3(\gamma)} \lesssim \|g\|_{\mathring{H}^{-1}(\gamma)} \leq |(\Pi_h^c \epsilon_h)^\ell|_{H^1(\gamma)}.
\end{equation}

Next, we have 
\begin{equation}\label{eq:H1_1}
\begin{aligned}
|(\Pi_h^c \epsilon_h)^\ell |_{H^1(\gamma)}^2 & = (g,(\Pi_h^c \epsilon_h)^\ell)_\gamma 
=(\Delta_\gamma^2 \phi,(\Pi_h^c \epsilon_h)^\ell)_\gamma
= (\nabla_\gamma\Delta_\gamma \phi,\nabla_\gamma(\Pi_h^c \epsilon_h)^\ell)_\gamma\\
& = (\nabla_\gamma\Delta_\gamma \phi,\nabla_\gamma(\Pi_h^c \epsilon_h)^\ell)_\gamma 
- (\nabla_{\Gamma_h}(\Delta_\gamma \phi)^e,\nabla_{\Gamma_h}\Pi_h^c \epsilon_h)_{\Gamma_h} \\
&\quad + (\nabla_{\Gamma_h}(\Delta_\gamma \phi)^e,\nabla_{\Gamma_h}(\Pi_h^c \epsilon_h-\epsilon_h))_{\Gamma_h}\\
&\quad -((\Delta_\gamma \phi)^e,\Delta_{\Gamma_h}\epsilon_h)_{\Gamma_h}+\sum_{e\in \mathcal{E}_h}\langle (\Delta_\gamma \phi)^e, [\nabla_{\Gamma_h}\epsilon_h \cdot \bm{n}] \rangle_e\\
&:= I_1+I_2+I_3+I_4.
\end{aligned}
\end{equation}

\textit{Estimates of $I_1$, $I_2$ and $I_4$.} Note that $\Delta_\gamma \phi, (\Pi_h^c \epsilon_h)^\ell \in H^1(\gamma)$, by \eqref{eq:trans-surface}, we have 
\begin{equation} \label{eq:H1-I1}
\begin{aligned}
|I_1| &= \big|((\textbf{R}_h - \textbf{I})\textbf{P} \nabla_\gamma(\Delta_\gamma \phi), \nabla_\gamma(\Pi_h^c \epsilon_h)^\ell)_\gamma\big| \lesssim h^2 \|\phi\|_{H^3(\gamma)} |\Pi_h^c \epsilon_h|_{H^1(\Gamma_h)} \\
&\lesssim h^2 \|\phi\|_{H^3(\gamma)} |\epsilon_h|_{H_h^1(\Gamma_h)} \lesssim h^2 \|\phi\|_{H^3(\gamma)} (\|u\|_{H^3(\gamma)} + \|f\|_{L^2(\gamma)}).
\end{aligned}
\end{equation}
Here, we have used Lemma \ref{lm:H1-relative} ($H^1$-conforming relative) and Theorem \ref{tm:energy-estimate} (energy norm error estimate). Next, we have 
\begin{equation} \label{eq:H1-I2}
|I_2| \lesssim h \| \phi \|_{H^3(\gamma)} |\epsilon_h|_{H_h^1(\Gamma_h)} \lesssim h^2 \| \phi \|_{H^3(\gamma)} (\|u\|_{H^3(\gamma)} + \|f\|_{L^2(\gamma)}).
\end{equation}
By the jump estimate \eqref{eq:jump-estimate1} and the energy estimate, we have 
\begin{equation} \label{eq:H1-I4}
|I_4|  \lesssim h\| \phi \|_{H^3(\gamma)} \triplenorm{\epsilon_h}_h \lesssim h^2 \| \phi \|_{H^3(\gamma)} (\|u\|_{H^3(\gamma)} + \|f\|_{L^2(\gamma)}).
\end{equation}

\textit{Estimate of $I_3$.} We denote $\phi_I := I_h \phi \in \mathring{V}_{h}$. From the FE scheme \eqref{eq:FEM}, it holds that 
$$ 
(\Delta_{\Gamma_h} u_h, \Delta_{\Gamma_h} \phi_I)_{\Gamma_h} + \sum_{e \in \mathcal{E}_h}h_e^{-1} \langle[\nabla_{\Gamma_h} u_h \cdot \bm{n}], [\nabla_{\Gamma_h} \phi_I \cdot \bm{n}]  \rangle_e = l_h(\phi_I),
$$ 
which leads to a decomposition of $I_3$ as 
$$
\begin{aligned}
I_3 = &-\big( (\Delta_\gamma \phi)^e-\Delta_{\Gamma_h}\phi_I,\Delta_{\Gamma_h}\epsilon_h \big)_{\Gamma_h} 
+ (\Delta_{\Gamma_h}\phi_I,\Delta_{\Gamma_h}u_h)_{\Gamma_h}-(\Delta_{\Gamma_h}\phi_I,\Delta_{\Gamma_h}u_I)_{\Gamma_h}\\
=& - \big( (\Delta_\gamma \phi)^e-\Delta_{\Gamma_h}\phi_I,\Delta_{\Gamma_h}\epsilon_h\big)_{\Gamma_h} 
+ l_h(\phi_I)-l(\phi)\\
&-\sum_{e\in \mathcal{E}_h}h_e^{-1} \langle[\nabla_{\Gamma_h} u_h \cdot \bm{n}], [\nabla_{\Gamma_h} \phi_I \cdot \bm{n}]  \rangle_e  \\
&+(\Delta_\gamma \phi,\Delta_\gamma u)_\gamma-(\Delta_{\Gamma_h}\phi_I,\Delta_{\Gamma_h}u_I)_{\Gamma_h}\\
:=&~J_1+J_2+J_3+J_4.
\end{aligned}
$$

By \eqref{eq:Delta-interpolation} and the energy estimate, there holds
\begin{equation} \label{eq:H1-J1}
|J_1|\lesssim h\|\phi\|_{H^3(\gamma)}\triplenorm{\epsilon_h}_h \lesssim h^2\|\phi\|_{H^3(\gamma)} (\|u\|_{H^3(\gamma)} + \|f\|_{L^2(\gamma)}).
\end{equation}
By source estimate \eqref{eq:source-estimate2} and \eqref{eq:H1-Delta}, there holds
\begin{equation} \label{eq:H1-J2-J4}
|J_2|+|J_4|\lesssim h^2 \|\phi\|_{H^3(\gamma)} (\|u\|_{H^3(\gamma)} + \|f\|_{L^2(\gamma)}).
\end{equation}
By jump estimates \eqref{eq:normal-jump} and \eqref{eq:energy-estimate-jump}, we have 
\begin{equation}\label{eq:H1-J3}
\begin{aligned}
|J_3| &\leq \Big( \sum_{e\in \mathcal{E}_h}h_e^{-1}  \| [\nabla_{\Gamma_h} u_h \cdot \bm{n}\|_{L^2(e)}^2 \Big)^{1/2}
\Big( \sum_{e\in \mathcal{E}_h}h_e^{-1}  \| [\nabla_{\Gamma_h} \phi_I \cdot \bm{n}\|_{L^2(e)}^2 \Big)^{1/2} \\
&\lesssim h^2 \|\phi\|_{H^3(\gamma)}( \|u\|_{H^3(\gamma)}+ \|f\|_{L^2(\gamma)}).
\end{aligned}
\end{equation}
Combine \eqref{eq:H1-J1}--\eqref{eq:H1-J3}, we have $|I_3| \lesssim h^2 \|\phi\|_{H^3(\gamma)}( \|u\|_{H^3(\gamma)}+ \|f\|_{L^2(\gamma)})$. Then, combining it with \eqref{eq:H1-I1}--\eqref{eq:H1-I4}, we obtain $|(\Pi_h^c \epsilon_h)^\ell |_{H^1(\gamma)}^2 \lesssim h^2 \|\phi\|_{H^3(\gamma)}( \|u\|_{H^3(\gamma)} + \|f\|_{L^2(\gamma)})$,
which yields 
\begin{equation} \label{eq:H1-Phce}
\|(\Pi_h^c \epsilon_h)^\ell \|_{H^1(\gamma)} \lesssim |(\Pi_h^c \epsilon_h)^\ell |_{H^1(\gamma)} \lesssim h^2( \|u\|_{H^3(\gamma)}+ \|f\|_{L^2(\gamma)})
\end{equation}
by using Poincar\'e inequality and  \eqref{eq:regularity-bound}.

To conclude, applying the property of $H^1$-conforming relative in \eqref{eq:sNZT-relative} gives $$\|\epsilon_h\|_{H_h^1(\Gamma_h)} \lesssim \|\Pi_h^c \epsilon_h \|_{H^1(\Gamma_h)} \lesssim \|(\Pi_h^c \epsilon_h)^\ell \|_{H^1(\gamma)}.$$ Utilizing \eqref{eq:H1-Phce} and the approximation result of $\|u^e - u_I\|_{H_h^1(\Gamma_h)}$ from \eqref{eq:I-L0}, along with the regularity result \eqref{eq:regularity}, completes the proof.
\end{proof}


\begin{remark}[estimate of $\nabla_{\gamma} u$]
In the broken $H^1$ norm error estimate, $\nabla_{\Gamma_h} u_h$ serves as an approximation to $\nabla_{\Gamma_h} u^e$. However, to approximate $\nabla_{\gamma} u$, it is necessary to map $\nabla_{\Gamma_h} u_h$ back to $\gamma$ using the Piola transform, denoted by $\mathscr{P}_{\bm{p}} \nabla_{\Gamma_h} u_h$. By leveraging the norm equivalence property of the Piola transform \eqref{eq:Piola-equivalence}, Lemma \ref{lm:Piola-derivative} ($\mathcal{O}(h^2)$ approximation of the Piola transform), and the broken $H^1$ norm error estimate \eqref{eq:H1-estimate}, we obtain:
\begin{align*}
\|\nabla_\gamma u - \mathscr{P}_{\bm{p}} \nabla_{\Gamma_h} u_h\|_{L^2(\gamma)} &\simeq \|\widebreve{\nabla_\gamma u} - \nabla_{\Gamma_h} u_h\|_{L^2(\Gamma_h)} \\
&\lesssim \|\widebreve{\nabla_\gamma u} - \nabla_{\Gamma_h} u^e\|_{L^2(\Gamma_h)} + \|u^e - u_h\|_{H^1_h(\Gamma_h)} \\
&\lesssim h^2 \|(\nabla_\gamma u)^e\|_{L^2(\Gamma_h)} + h^2 \|f\|_{L^2(\gamma)} \\
&\lesssim h^2 \|f\|_{L^2(\gamma)}.
\end{align*}
It is worth noting that directly approximating $(\nabla_{\gamma} u)^e$ with $\nabla_{\Gamma_h} u_h$ would only yield first-order convergence, as discussed in Remark \ref{rm:Piola-vs-extension}.
\end{remark}

\section{Numerical experiments} \label{sec:numerical}
We perform several numerical experiments to validate the theoretical results, with the implementation based on the software iFEM \cite{chen2008ifem}. The first two model problems, following those in \cite{olshanskii2009finite}, involve equations on a sphere and a torus, while the third problem is posed on a more general, implicitly defined surface. For each case, we refine the surface meshes by subdividing each triangular element into four smaller triangles and projecting the newly created nodes back onto the surface.

To facilitate computation, we report the following errors on a discrete surface:
\begin{equation*}
\begin{aligned}
&E_0 = \|u^e-u_h\|_{L^2(\Gamma_h)},\quad &&E_1=\|\nabla_{\Gamma_h}u^e-\nabla_{\Gamma_h} u_h\|_{L^2(\Gamma_h)}, \\
&E_{\Delta} = \|(\Delta_{\gamma} u)^e-\Delta_{\Gamma_h} u_h\|_{L^2(\Gamma_h)},\quad &&E_{\text{jump}}=\sum_{e\in \mathcal{E}_h}h_e^{-1} \|[\nabla_{\Gamma_h} u_h \cdot \bm{n}]\|_{L^2(e)}^2.
\end{aligned}
\end{equation*}


\subsection{Example 1: Problem on the Unit Sphere}

In the first example, we consider a unit sphere with radius $r = 1$. The source term $f$ is chosen such that the exact solution is $u = r^{-3}(3x^2y - y^3)$. This function $u$ is an eigenfunction of $-\Delta_\gamma$ satisfying
\[
-\Delta_\gamma u = 12u, \quad f = (\Delta_\gamma)^2 u = 144(3x^2y - y^3) \quad \text{on } \gamma.
\]
The errors show that the convergence order is 2 in both the broken $H^1$ norm, as predicted by Theorem \ref{tm:H1-estimate}. Additionally, $E_\Delta$ and $E_{\text{jump}}$ converge at a rate of 1, consistent with Theorem \ref{tm:energy-estimate2}.

\begin{table}[!htbp]
	\centering
	\begin{tabular}{@{}lcccccccc@{}}
		\toprule
		Dof&$E_0$&order&$E_1$&order&$E_\Delta$&order&$E_{\text{jump}}$&order\\
		\midrule
		   486&7.54e-02& &3.14e-01& &2.11e-00& &5.06e-01&  \\
		  1926&1.91e-02& 1.98& 7.96e-02& 1.98 &1.03e-00& 1.03& 2.61e-02& 0.96\\
		  7686&4.78e-03& 2.00& 1.99e-02& 2.00 &5.13e-01& 1.01& 1.31e-02& 0.99\\
		 30726&1.19e-03& 2.00& 4.99e-03& 2.00 &2.56e-01& 1.00& 6.58e-02& 1.00\\
		122886&2.99e-04& 2.00& 1.25e-03& 2.00 &1.28e-01& 1.00& 3.29e-02& 1.00\\
		\bottomrule
	\end{tabular}
	\caption{Error table for surface NZT element approximation on the unit sphere.}
	\label{tab:sphere}
\end{table}

\begin{figure}[!htbp]
	\centering
	\includegraphics[scale=0.6]{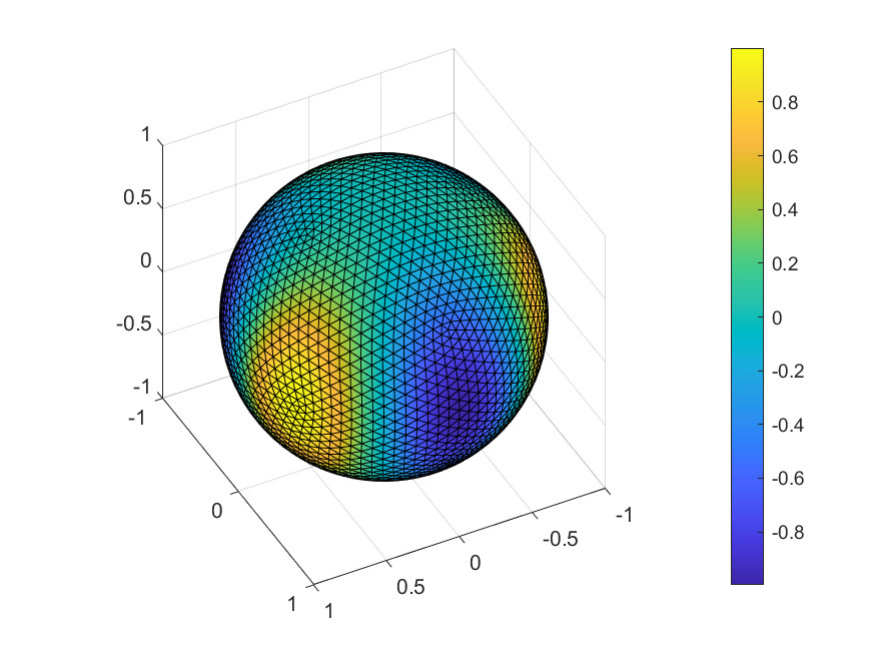}
	\caption{The numerical approximation of $u = r^{-3}(3x^2y - y^3)$ with 7686 DoFs on the unit sphere.}\label{fig:sphere}
\end{figure}

\subsection{Example 2: Problem on a torus}
The second problem is formulated on a torus with a major radius $R = 1$ and a minor radius $r = 0.6$. The toroidal coordinates $\{\rho, \theta, \phi\}$ are defined such that the Cartesian coordinates are given by
$$
x = (R + \rho \cos(\theta)) \cos(\phi), \quad y = (R + \rho \cos(\theta)) \sin(\phi), \quad z = \rho \sin(\theta),
$$
where $0 \leq \theta < 2\pi$, $0 \leq \phi < 2\pi$, and $0 \leq \rho$. The signed distance function of the torus is $d = \rho - r$. The source term $f$ is implemented based on the code in \cite{larsson2017continuous}, and the exact solution for this problem is $u = \sin(3\phi) \cos(3\theta + \phi)$. 

Applying the Laplace-Beltrami operator to $u$, we obtain:
\begin{align*}
-\Delta_\gamma u =& \; r^{-2} 9\sin(3\phi) \cos(3\theta + \phi) \\
& - r^{-1}(R + r \cos(\theta))^{-1} 3 \sin(\theta) \sin(3\phi) \sin(3\theta + \phi) \\
& + (R + r \cos(\theta))^{-2} \big(10 \sin(3\phi) \cos(3\theta + \phi) + 6 \cos(3\phi) \sin(3\theta + \phi)\big) \quad \text{on } \gamma.
\end{align*}

The numerical solution is visualized in Figure \ref{fig:torus}, and the errors are provided in Table \ref{tab:torus}. The observed rates of convergence align well with theoretical expectations.

\begin{table}[!htbp]
	\centering
	\begin{tabular}{@{}lcccccccc@{}}
		\toprule
		Dof&$E_0$&order&$E_1$&order&$E_\Delta$&order&$E_{\text{jump}}$&order\\
		\midrule
		  1536&7.92e-01&     & 4.25e-00&     & 3.99e+01&     & 8.99e-00&     \\
		  6144&2.26e-01& 1.81& 1.49e-00& 1.51& 2.15e+01& 0.89& 6.88e-00& 0.39\\
		 24576&6.88e-02& 1.72& 4.29e-01& 1.80& 1.13e+01& 0.93& 3.91e-00& 0.81\\
		 98304&1.73e-02& 1.99& 1.09e-01& 1.98& 5.83e-00& 0.95& 2.02e-00& 0.95\\
		393216&4.24e-03& 2.02& 2.67e-02& 2.03& 2.96e-00& 0.98& 1.04e-00& 0.96\\
		\bottomrule
	\end{tabular}
	\caption{Error table for surface NZT element approximation on the torus.}
	\label{tab:torus}
\end{table}

\begin{figure}[!htbp]
	\centering
	\includegraphics[scale=0.6]{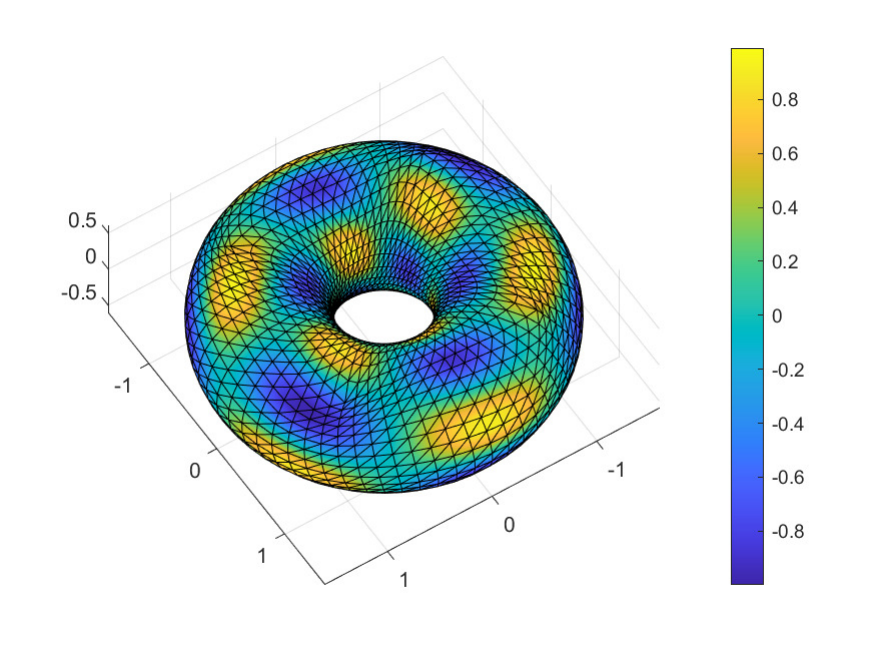}
	\caption{The numerical approximation of $u = \sin(3\phi) \cos(3\theta+\phi)$ with 6144 DoFs on the torus.}\label{fig:torus}
\end{figure}

\subsection{Example 3: Problem on an implicitly defined surface}
The third example, adapted from \cite{dziuk1988finite}, is defined on a general surface implicitly represented by the level set function
$$
\phi(x, y, z) = (x - z^2)^2 + y^2 + z^2 - 1.
$$
The function $f$ is chosen such that the exact solution is $u = y$, and the numerical approximation is illustrated in Figure \ref{fig:gen}. The expressions for $f$ and $\Delta_\gamma u$ are computed exactly using MATLAB's symbolic computation toolbox. In this example, obtaining explicit expressions for $d(x)$, $\bm{p}(x)$, and $\textbf{H}(x)$ is not straightforward. Therefore, the first-order projection algorithm from \cite{demlow2007adaptive} is utilized to iteratively determine $\bm{p}(x)$. To circumvent the computation of $\textbf{H}(x)$, we report the error
$$
E_1^\star = \|\textbf{P}_h(\nabla_{\gamma} u)^e - \nabla_{\Gamma_h} u_h\|_{L^2(\Gamma_h)},
$$
instead of $E_1$. Given that $\nabla_{\Gamma_h} u^e = \textbf{P}_h(\textbf{P} - d \textbf{H})(\nabla_{\gamma} u)^e$, the difference satisfies $|E_1 - E_1^\star| = \mathcal{O}(h^2)$. The corresponding error results are presented in Table \ref{tab:gen}, where the second-order convergence of $E_1^\star$ sufficiently demonstrates the second-order convergence of $E_1$.


\begin{table}[!htbp]
	\centering
	\begin{tabular}{@{}lcccccccc@{}}
		\toprule
		Dof&$E_0$&order&$E_1^\star$&order&$E_\Delta$&order&$E_{\text{jump}}$&order\\
		\midrule
		  3462&5.50e-01&     & 8.17e-01&     & 2.15e-00&     & 1.33e-01&     \\
		 13830&1.88e-01& 1.55& 2.81e-01& 1.54& 1.30e-00& 0.72& 8.04e-02& 0.73\\
		 55302&5.20e-02& 1.85& 7.84e-02& 1.84& 6.99e-01& 0.90& 4.17e-02& 0.95\\
		221190&1.31e-02& 1.99& 1.98e-02& 1.98& 3.52e-01& 0.99& 2.06e-02& 1.02\\
		884742&3.27e-03& 2.01& 4.93e-03& 2.01& 1.76e-01& 1.00& 1.02e-02& 1.02 \\
		\bottomrule
	\end{tabular}
	\caption{Error table for surface NZT element approximation on an implicitly defined surface.}
	\label{tab:gen}
\end{table}

\begin{figure}[!htbp]
	\centering
	\includegraphics[scale=0.6]{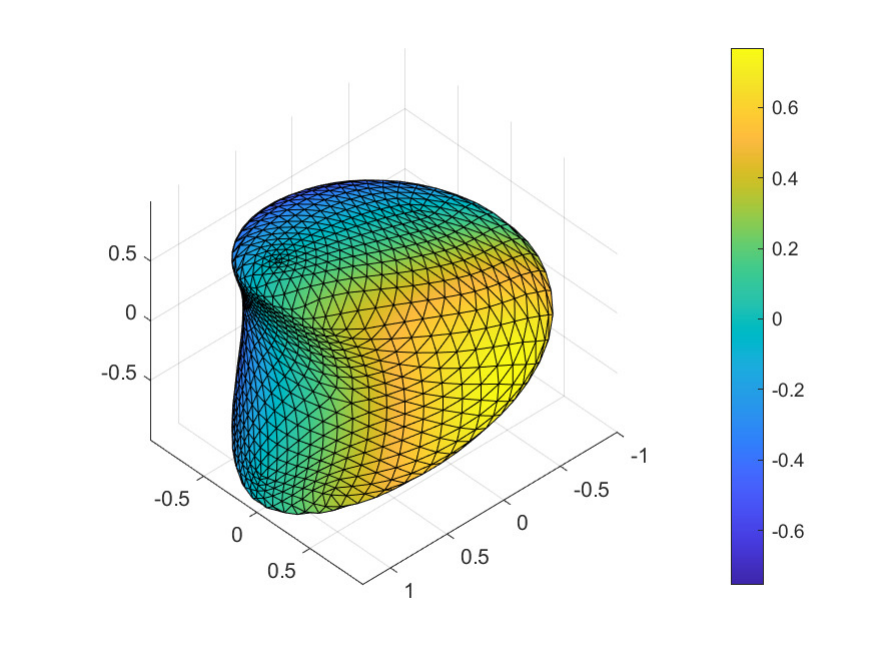}
	\caption{The numerical approximation of $u = y$ with 3462 DoFs on an implicitly defined surface.}
	\label{fig:gen}
\end{figure}


\bibliographystyle{siamplain}
\bibliography{surfaceNZT.bib} 
\end{document}